\documentclass[12pt, leqno]{article}

\usepackage[a4paper, margin=1in]{geometry}
\usepackage[vietnamese, UKenglish, british]{babel}
\usepackage[utf8]{inputenc}
\usepackage[T1]{fontenc}
\usepackage{vntex}
\usepackage[en-US]{datetime2}
\usepackage{tabularx}
\usepackage[intlimits]{amsmath}
\usepackage{amssymb}
\usepackage{amsthm}
\usepackage{bm}
\usepackage{mathtools}
\usepackage{mathrsfs}
\usepackage{graphicx}
\usepackage{tikz}
	\usetikzlibrary{patterns, calc}
	\tikzset{CENTRE/.style={%
execute at end picture={%
\path let \p1=(current bounding box.west),\p2=(current bounding box.east)
in ({-max(-1*\x1,\x2)},\y1) ({max(-1*\x1,\x2)},\y1);
}}}
\usepackage{enumerate}
\usepackage{mathrsfs}
\usepackage{longtable}
\usepackage{tabu}
\usepackage{fancyhdr}
\usepackage{xurl}
\usepackage{paralist}
\usepackage[]{hyperref}
	\hypersetup{colorlinks  = true, breaklinks=true, pdftitle={Delayed interior area-to-height estimate for CSF}}
\usepackage{float}
\usepackage{xcolor}
\usepackage{cprotect}
\usepackage{listings}
\usepackage{multicol}
\usepackage{subcaption}
\usepackage{xassoccnt}
\usepackage[type={CC}, modifier={by}, version={4.0}, imageposition={left}, imagewidth={7em}, imagedistance={1.5em}]{doclicense}

\newcommand{\mc}[1]{\mathcal{#1}}
\newcommand{\mb}[1]{\mathbb{#1}}

\newcommand{\msc}[1]{\mathscr{#1}}
\newcommand{\mfr}[1]{\mathfrak{#1}}

\newcommand{\set}[1]{ \left\{ #1 \right\} }

\newcommand{\mr}[1]{\mathrm{#1}}

\DeclareMathOperator{\graph}{graph}

\newcommand{\ms}{\begin{equation}}
\newcommand{\mf}{\end{equation}}
\newcommand{\msa}[1]{\begin{equation}\label{#1}\begin{aligned}}
\newcommand{\mfa}{\end{aligned}\end{equation}}

\newcommand{\pd}{\partial}
\newcommand{\dbyd}[2]{\frac{\mr{d}#1}{\mr{d}#2}}

\newcommand{\veps}{\varepsilon}

\newcommand{\eds}[1]{\, \mr{d}#1}

\newcommand{\Cinfloc}{\mr{C}^\infty_\mr{loc}}
\newcommand{\Czloc}{\mr{C}^0_\mr{loc}}

\DeclareMathOperator{\supp}{supp}

\title{A delayed interior area-to-height estimate for the Curve Shortening Flow}
\author{Arjun Sobnack\footnote{Department of Mathematics, The University of Texas at Austin, TX 78712, USA. \emph{Work completed at the Warwick Mathematics Institute, University of Warwick, Coventry, CV4 7AL, UK.} \\ \indent Email: \href{mailto:arjun.sobnack@austin.utexas.edu}{arjun.sobnack@austin.utexas.edu} \\ \indent Website: \href{https://sites.google.com/view/sobnack}{sites.google.com/view/sobnack}}}
\date{31\textsuperscript{st} March 2026}

\DeclareCoupledCounters[name=eqnfig]{figure, equation}
\numberwithin{equation}{section}
\numberwithin{figure}{section}

\theoremstyle{plain}

\newtheorem{lem}[equation]{Lemma}

\newtheorem{thm}[equation]{Theorem}
\newtheorem{prop}[equation]{Proposition}

\newtheorem{claim}{Claim}
\newtheorem*{claim*}{Claim}
\newtheorem*{thm*}{Theorem}
\theoremstyle{definition}
\newtheorem{defn}[equation]{Definition}
\newtheorem{rem}[equation]{Remark}

\graphicspath{ {./images_exists_meas/} }

\allowdisplaybreaks

\begin{document}
\selectlanguage{british}

\maketitle

\begin{abstract} 
The principle of \emph{delayed  parabolic regularity} for the Curve Shortening Flow---that if two evolving curves bound a region of area $\mc{A}$, then, starting from time $\mc{A}/\pi$, the regularity of one curve is controllable in terms of the time elapsed, the area $\mc{A}$ and the regularity of the other curve---was proposed by Topping \& the author in \cite{sobtopdelayed}, where they also provided a number of graphical situations in which their delayed regularity framework is valid.
In this paper, we generalise some of the results in \cite{sobtopdelayed} within the graphical setting, ultimately by showing that there holds an \emph{interior} graphical estimate for the Curve Shortening Flow in the spirit of the proposed framework. 
We also provide a few applications of our estimate, such as the existence of Graphical Curve Shortening Flows starting {weakly} from Radon measures without point masses.
\end{abstract}

\section{Introduction}\label{SEC:intro}

In its classical form, the \emph{Curve Shortening Flow (CSF)} takes a smooth embedding of the circle (a smooth \emph{loop}) in the Euclidean plane and evolves it in the direction of its geodesic curvature vector.
The CSF is the instance of the \emph{Mean Curvature Flow (MCF)} in the lowest non-trivial dimensions; in particular, it arises as the \( \mr{L}^2 \)--gradient flow of the length functional.  

Following on from {Huisken}'s seminal paper (\cite{huisken1984flow}) on the MCF of smooth, closed and mean-convex hyper-surfaces in \( \mb{R}^{n + 1} \) for $n \geq 2$, Gage \& Hamilton tackled the question of the long-time behaviour of the CSF, i.e.~the case of $n=1$.
In \cite{gage1986heat} they showed that the evolution \( [0, T) \ni t \mapsto \gamma(t) \subset \mb{R}^2 \) of a given \emph{convex} smooth loop \( \gamma_0 =: \gamma(0)   \) exists for a maximal time \( T < \infty  \) dictated only by the area of the region enclosed by \( \gamma_0 \). 
Moreover, they precisely described the behaviour of the evolution of the loop as time approaches \( T \): the curve \( \gamma(t) \) smoothly approaches a circle of radius \( {\sqrt{2(T - t)}} \) in the sense of blow-ups as \( t \nearrow T \).
Shortly after, Grayson showed in \cite{grayson1987heat} that the evolution of an \emph{arbitrary} smooth loop in the Euclidean plane becomes convex before encountering any singularities, at which point Gage \& Hamilton's work 
applies.

Collectively, \cite{gage1986heat, grayson1987heat} is known and celebrated as the Gage--Grayson--Hamilton Theorem. 
Grayson subsequently weakened, in \cite{grayson1989shortening},  the assumption of planarity of the initial curve, an extension often referred to as Grayson's Theorem.
\\
 
The CSF has been studied in a variety of different scenarios.
The focus of this paper is the setting of CSFs which are, at each time, expressible as graphs over a fixed line or fixed line segment. 
In this setting, the weakly parabolic system for the evolution of a curve reduces to a parabolic equation for the \emph{graphing function}.

Some of the pioneering works in this direction are \cite{ecker1989mean, ecker1991interior} due to Ecker \& Huisken. 
They showed that, given any\footnote{%
Ecker \& Huisken's work considered the more general case of the MCF of graphical hyper-surfaces.%
} graph \( \gamma_0 = \mr{graph}(u_0) \subset \mb{R}^2 \) of a locally Lipschitz function \( u_0 : \mathbb{R} \mapsto \mathbb{R} \), there exists a CSF \( (0, \infty) \ni t \mapsto \gamma(t) = \graph(u(\, \cdot \, , t)\!) \subset \mb{R}^2 \), defined for all time, which attains the initial datum \( \gamma_0  \) locally uniformly, and which remains graphical over the same line; we call this flow the the \emph{Ecker--Huisken flow of \( u_0 \)}. 
They achieved their existence theory by appealing directly to the \emph{Graphical Curve Shortening Flow (GCSF)} equation, which the graphing function \(  u : \mathbb{R} \times (0, T) \mapsto \mathbb{R}\) necessarily satisfies. 
A key step in their proof is to establish \emph{a priori} estimates for the GCSF equation of smooth graphing functions, which propagate an initial local Lipschitz bound forward to the CSF, provided one moves in a little.
Such bounds are stable under locally uniform limits, and thus allowed Ecker \& Huisken to construct solutions via an approximation procedure. 

Around the same time, Angenent developed his \emph{Intersection Principle} (see \cite{angenent1988zero}) and used it in \cite{angenent1991parabolic} to construct solutions to the CSF starting from Lipschitz loops in the uniform sense. 
As in Ecker \& Huisken's work, a key step is to establish an \emph{a priori} estimate for smooth loops which propagates a Lipschitz bound forward to the flow.
Angenent introduced a new and remarkably geometric technique to establish his estimates, which has found much success in other situations---see, for example, \cite{ilmanen1992generalized, chou2001quasilinear, clutterbuck2004thesis, nagase2005interior, bourni2023compact}, and more. 
In particular, Topping \& the author used Angenent's technique in \cite{sobtopdelayed} to establish sharp \emph{delayed} estimates for the GCSF, which fit into the framework to be recalled in \S\ref{sec:framework}.   

Chou \& Zhu used (amongst other insights) a combination of Ecker \& Huisken and Angenent's work to  establish, in \cite{chou1998shortening}, that the CSF can be started, in the locally uniform sense, from any properly embedded and locally Lipschitz planar curve
 which splits the plane into two regions of infinite area, and that their resulting flow exists for all time. 
Their result is the state-of-the-art in terms of the existence theory for non-compact and non-graphical curves. 
They also proved the state-of-the-art uniqueness result for CSFs that evolve through curves whose \emph{ends remain graphical}.
Given a curve \( \gamma_0 \subset \mb{R}^2 \) (of the regularity considered in \cite{chou1998shortening}) with graphical ends, we call the unique solution to the CSF through curves whose ends remain graphical 
the \emph{Chou--Zhu flow of \( \gamma_0 \)}. \\

The running theme in the works \cite{ecker1989mean, ecker1991interior, angenent1991parabolic} above is that a strong regularity condition is shown to be preserved by the CSF. 
In the past three decades, significant progress has been made in the graphical situation towards \emph{instantaneous} estimates, which take a weak initial regularity constraint and deduce a strong regularity condition on the CSF for strictly positive times. 
Evans \& Spruck established in \cite{evans1992motion} (see also \cite{colding2004sharp}) that an $\mr{L}^\infty_\mr{loc}$--bound on an initial graphing function induces an instantaneous $\mr{L}_\mr{loc}^\infty$--bound on the gradient of the subsequent flow, provided one moves in a little.
Their result was 
later extended, independently, by Clutterbuck and Nagase \& Tonegawa in \cite{clutterbuck2004thesis} and \cite{nagase2005interior} respectively; 
in particular, Clutterbuck used her generalisation to establish an existence theory for a family of parabolic equations which contains graphical instances of the CSF and MCF.    
For $\mr{L}^{p}_\mr{loc}$--bounds when $1 < p < \infty$, Chou \& Kwong established in \cite{chou2020general} an \emph{a priori} estimate which takes an initial $\mr{L}^{p}_\mr{loc}$--bound on a graphing function instantaneously to a $\mr{L}^\infty_\mr{loc}$--bound on the gradient of the subsequent flow, provided one moves in a little; they used their estimate to establish the existence of GCSFs starting from graphs of \(\mr{L}^{p}_\mr{loc} \)--initial data.
They also established a uniqueness result from \( \mr{L}^{p} \)--initial data, building on Chou \& Zhu's uniqueness theory (\cite{chou1998shortening}); see also \cite{daskalopoulos2023uniqueness}.

For the non-graphical case, see, for example, the works of Angenent (\cite{angenent1990parabolic}; loosely, initial curves whose curvature is $\mr{L}^p$ for some $p \gg 1$) and Lauer (\cite{lauer2013new}; finite-length Jordan curves). \\

One of the aims of this paper is to extend the existence theory to include, loosely, graphs of Radon measures without point masses. 
In some sense, the strategy we follow is familiar: we first establish an \emph{a priori} estimate for smooth initial data which is stable under \emph{weak convergence of measures}, and then use an approximation procedure to pass to a non-smooth limit.
However, the \emph{a priori} estimate we establish is pronouncedly \emph{unfamiliar} in comparison to the previously-mentioned estimates in \cite{gage1986heat, ecker1989mean, ecker1991interior, angenent1990parabolic, angenent1991parabolic, evans1992motion, lauer2013new} etc., 
as well as, to the author's best knowledge, all other estimates (except, of course, for those in \cite{sobtopdelayed}) from the CSF literature,   
which, recall, start to hold \emph{instantaneously}.
Our estimate is a \emph{delayed} parabolic \emph{a priori} estimate, and in our setting it is, in fact, \emph{impossible} for the estimate to hold \emph{before} a so-called \emph{magic time}, which is determined completely by the `area' under the initial graph.
In particular, we apply our estimate slightly differently to the applications in, for example, \cite{ecker1989mean, ecker1991interior, angenent1990parabolic, angenent1991parabolic, lauer2013new,  clutterbuck2004thesis, chou2020general}. 

Our estimates are reminiscent of those established by Topping \& Yin in \cite{topping2017sharp} for the two--dimensional \emph{Ricci Flow (RF)}, and indeed there is a bridge between their work and ours---see Remark \ref{rem:ty17}. 
Consequently, our existence theory is analogous to that of Topping \& Yin's (\cite{topping2021smoothing}) for the RF.

\subsection{A novel framework}\label{sec:framework}

In \cite{sobtopdelayed} (see also \cite{sobtopmodulus}), Topping \& the author proposed a version of the following framework by which the CSF regularises a curve:

\begin{quote}
\it Given two smooth families  \( [0, T) \ni t \mapsto \gamma^-(t)  \subset \mb{R}^2 \) and \( [0, T) \ni t \mapsto \gamma^+(t) \subset \mb{R}^2 \) of properly embedded curves in the Euclidean plane evolving under the Curve Shortening Flow, 
for which \( \gamma^-(t) \) and \( \gamma^+(t) \) are disjoint, are either both compact or both non-compact (for all \( t \in [0, T) \)) and bound an evolving connected region of fixed and finite area \( \mc A \), 
define the \emph{magic time} to be \(  t_\star := \mc{A}/\pi \).
Then, after flowing beyond the magic time \( t_\star  \), 
one expects to start being able to control the regularity of \( \gamma^+(t) \) for times \(  t_\star < t \ll T \) in terms of the time \( t - t_\star \) elapsed, the area \( \mc A \) and the regularity of \( \gamma^-(t) \).
\end{quote}

Several situations in which the framework above holds exactly as stated were provided in \cite{sobtopdelayed}. 
Most relevant to this paper is the case where \( t \mapsto \gamma^-(t) \) is the static flow of the \( x \)--axis:

\begin{thm}[Delayed \( \mr{L}^1 \)-to-\( \mr{L}^\infty \) estimate; {\cite[Theorem 1.5]{sobtopdelayed}}]
\label{thm:globheightest}
Let \( u : \mathbb{R} \times [0, \infty) \mapsto \mathbb{R}_{\geq 0}\) be a smooth {\normalfont positive} Graphical Curve Shortening Flow starting from \( u_0 := u(\, \cdot \, , 0) \in \mr{L}^1(\mb{R}) \).
Let
\ms
\mc{A} := \| u_0 \|_{\mr{L}^1(\mb{R})} = \int_{x \in \mb{R}} u_0(x) \, \mr{d}x 
\mf
and define the {\normalfont magic time} to be 
\ms
t_\star := \frac{\mc{A}}{\pi}.
\mf
Then there exists a universal (that is, {\normalfont independent} of $u$) and decreasing function \( H : (\frac{1}{\pi}, \infty) \mapsto (0, \infty) \) for which
\ms
\| u(\,\cdot\,,t) \|_{\mr{L}^\infty(\mb{R})} \leq \sqrt{t}\cdot  H \left( \frac{t}{\mc A} \right)
\mf
for all \( t > t_\star \).
\end{thm}

Recall that a GCSF is a smooth function \( u : J \times (0, T) \mapsto \mathbb{R} \) which satisfies the quasi-linear parabolic equation 
\ms
\tag{GCSF} 
\pd_t u = \frac{1}{1 + (\pd_x u)^2} \pd^2_x u = \pd_x [ \tan^{-1}(\pd_x u) ],  \label{GCSF}
\mf 
where \( J \subseteq \mb{R} \) is an interval.
A solution to \eqref{GCSF} induces a family of curves \( (0, T) \ni t \mapsto \gamma(t)  = \graph(u(\,\cdot\,,t)\!) \) which solve the CSF equation
\ms
\tag{CSF} 
\left\langle \pd_t \gamma , \bm{n} \right\rangle = \kappa,  \label{CSF}
\mf
where the geodesic curvature vector  \( \bm{\kappa}(t) = \kappa(t) \bm{n}(t) \)  of \( \gamma(t) \) decomposes as the product of the geodesic curvature \( \kappa(t) \) with the unit normal \( \bm{n}(t) \).
 Conversely, given a solution to \eqref{CSF} which happens to be graphical, its graphing function solves \eqref{GCSF}. 
 
 When an initial condition, say \( u_0 =: u(\,\cdot\,,0) \) for \eqref{GCSF} or \( \gamma_0 =: \gamma(0) \) for \eqref{CSF}, is imposed, the regularity of the initial datum determines the mode in which the flow converges backward in time. 
 For example, in Chou \& Kwong's work \cite{chou2020general}, since \( u_0 \in \mr{L}^{p}_\mr{loc}(\mb{R}) \) for some $p > 1$, they insist that \( u(\,\cdot\,,t) \to u_0 \in \mr{L}^{p}_\mr{loc}(\mb{R}) \) as \( t \searrow 0 \). 
 In the case that \( \gamma_0 \) is continuous, we write \( \gamma : J \times [0, T) \mapsto \mb{R}^2 \) or \( [0, T) \ni t \mapsto \gamma(t) \subset \mb{R}^2 \).
 In the case that \( u_0 \) is continuous, we write \( u \in \Czloc(\mb{R} \times [0, T)\!) \).
 And in the case that \( u_0 \) is smooth, we write \( u \in \Cinfloc(\mb{R} \times [0, T)\!) \).  \\

In \cite{sobtopdelayed}, Topping \& the author also provided the following construction, which shows that the proposed \emph{magic time} \( t_\star = \mc{A}/\pi \) is sharp:

\begin{thm}[{Variant of \cite[Theorems 1.1 and 1.9]{sobtopdelayed}; see \cite[Theorem 5.2.21]{sobphd}}]\label{thm:dirdel}
There exists a sequence \( (v_n)_{n \in \mb{N}} \subset \Cinfloc(\mb{R} \times [0, \infty); \mb{R}_{\geq 0}) \) of Graphical Curve Shortening Flows, with
\ms
\| v_n(\,\cdot\,,0) \|_{\mr{L}^1(\mb{R})} = 1
\mf
for each \( n \in \mb{N} \), such that, for any \(0 \leq s \leq t_\star = 1/\pi < t < \infty \), there holds 
\ms 
\sup_{n \in \mb{N}} v_n(0, s) = \infty \qquad \text{ and } \qquad \sup_{n \in \mb{N}} \| v_n(\,\cdot\,,t) \|_{\mr{L}^\infty(\mb{R})} < \infty.
\mf
\end{thm}

As applications of Theorem \ref{thm:globheightest}, Topping \& the author gave a new proof of an instantaneous \( \mr{L}^{p}$-to-\( \mr{L}^\infty \)--estimate (\cite[Theorem 1.10]{sobtopdelayed}), where $1 < p < \infty$, and established the existence of GCSFs starting from \( \mr{L}^1 \)--initial datum (\cite[Theorem 1.11]{sobtopdelayed}).
The existence result in this paper is a generalisation of the existence result of Topping \& the author, but uses a slightly different method.

\subsection{Main results}\label{sec:results}

Throughout this paper, we use the notation \( J_r(a) \) to mean the interval \( (a-r,a+ r) \subset \mb{R} \) thought of as a spatial domain. 
We also abbreviate \( J_r(0) \) to \( J_r \). 
Any instance of a \( J \) without the decoration as before  will refer to an interval of unspecified length, centre or closedness. 
This is done purely to distinguish the spatial and temporal domains at the level of notation.  

Given a spatial domain $\Omega \subseteq \mb R$, for each $k \in \mb{N}_0 \cup \set{\infty}$ we denote by $\mr{C}^k(\Omega)$ the space of $k$--times differentiable functions all of whose derivatives are bounded and uniformly continuously, and for each pair  $(k,\alpha) \in \mb{N}_0 \times (0, 1]$ we denote by $\mr{C}^{k,\alpha}(\Omega) \subset \mr{C}^k(\Omega)$ the sub-space of functions whose $k$\textsuperscript{th} derivative is $\alpha$--H\"{o}lder continuous. 
We denote by $\mr{C}^k_\mr{loc}(\Omega)$ and $\mr{C}^{k,\alpha}_\mr{loc}(\Omega)$ their respective local spaces.
For a space-time domain $\mathcal{D}$, for example $\mathcal{D} = J_1 \times (0, T)$, we denote by $\mr{C}^{k;l}(\mc{D})$ and $\mr{C}^{k,\alpha; l, \beta}(\mathcal{D})$ the corresponding parabolic versions of the respective spaces before, and by $\mr{C}^{k;l}_\mr{loc}(\mc{D})$ and $\mr{C}^{k,\alpha; l, \beta}_\mr{loc}(\mathcal{D})$ their localised counterparts. 
This notation is consistent with the notation introduced in \S\ref{sec:framework} to specify the mode in which a GCSF attains it initial datum.

Given any function \( f : \Omega \times \mathcal{I} \mapsto \mb{R} \) of space-time, where \( \mathcal I \) is one of $(0, T)$, $[0, T)$, $(0, T]$ or $[0, T]$, we denote the $t$--time slice of \(f\),  for each $t \in \mathcal I$,  by \( f(t) := f(\,  \cdot \, , t ) : \Omega \mapsto \mathbb{R}  \).  
If \( \mc X \) is a function space for which \( f(t) \in \mc X \) for each \( t \in \mc I \), and the mapping \( \mc I \ni t \mapsto f(t) \in \mc X \) is continuous, then we write \( f \in \mr{C}^0_\mr{loc}(\mc I ; \mc X) \).
 \\

In this paper, we generalise some of the results in \cite{sobtopdelayed}---namely Theorem \ref{thm:globheightest}, the \( \mr{L}^{p} \)-to-\( \mr{L}^\infty \)--estimate and the \( \mr{L}^1 \)--existence result---by demonstrating that a local, but still graphical, version of Theorem \ref{thm:globheightest} holds. 
The crucial new localised estimate is

\begin{thm}[Delayed \( \mr{L}^1_\mr{loc} \)-to-\( \mr{L}^\infty_\mr{loc} \)--estimate] \label{thm:intheightest}
Let \( u \in \Cinfloc(J_1 \times [0, T) \! ) \) be a Graphical Curve Shortening Flow starting from \( u_0 := u(0) = u(\, \cdot \, , 0) \in \mr{L}^1(J_1) \), and suppose that
\ms
t_\star : = \frac{\| u_0 \|_{\mr{L}^1(J_1)}}{\pi} =  \frac{1}{\pi} \int_{x \in J_1} |u_0(x)| \eds x < T.
\mf
Then for all \( \delta > 0 \), there exists a constant \( C = C(\delta) < \infty \)  such that the bound
\ms
|u( 0,  t )| \leq C + \frac{1}{2} \| u_0 \|_{\mr{L}^1(J_1)} + \frac{\pi t}{2}  \label{uupbdd}
\mf
holds for all \( t < T \) satisfying 
\ms 
t \geq (1+ \delta) \cdot  t_\star. 
\mf
\end{thm}

The construction of Theorem \ref{thm:dirdel} also applies in the local setting of Theorem \ref{thm:intheightest}, and consequently the \emph{magic time} \( t_\star = \| u_0 \|_{\mr{L}^1(J_1)} / \pi \) is sharp here too.
We prove Theorem \ref{thm:intheightest} in \S\ref{SEC:locgradest}. 

For generalisations of a different flavour, the interested reader is directed to, for example, \cite{sobtopharnack}. 
Topping \& the author expect \cite{sobtopmodulus} to play a role in further generalisations. 
 \\

Once a height bound has been established, Evans \& Spruck's estimate (\cite{evans1992motion}) implies that an interior gradient bound holds an instant later, at which point standard parabolic regularity theory (see e.g.~\cite{ladyzenskaja1988linear, krylov1996lectures}) can be used to deduce that higher-order bounds hold after one more instant.  
For details, see \cite{sobphd}. 
We shall use this argument in the proof of our existence result, which will be stated momentarily.

One advantage of the local estimate \eqref{uupbdd} from Theorem \ref{thm:intheightest} is that it is applicable to CSFs that are only \emph{locally} graphical.
Specifically, provided one ensures that an arc of a CSF remains a graph over a fixed line segment for long enough, our estimate implies local gradient bounds, and consequently higher-order derivative bounds.
In particular, one can deduce higher-order covariant derivative bounds on the geodesic curvature, which are by nature geometric as well as analytic.   
\\

Our delayed interior area-to-height estimate (Theorem \ref{thm:intheightest}) is a consequence of the \emph{Harnack inequality} for our local \emph{Harnack quantity} from \S\ref{SEC:harnquan}, which is adapted from the global Harnack quantity of Topping \& the author in \cite{sobtopdelayed}.
Recall that the global Harnack quantity in \cite{sobtopdelayed} is used more-or-less directly on the class of GCSFs for which the height estimates of Theorem \ref{thm:globheightest} are claimed.
By constrast, our local {Harnack quantity} is \emph{not} defined for general GCSFs, but is instead defined on the subclass of \emph{local GCSFs}.
Our local GCSFs will act as  barries for our general GCSFs, and, in particular, height bounds obtained via our Harnack inequality on the former will imply the height bounds desired by Theorem \ref{thm:intheightest} on the latter. \\

As an application of Theorem \ref{thm:intheightest}, in \S\ref{SEC:Lpest} we relatively swiftly deduce an \emph{instantaneous} \( \mr{L}^{p}_\mr{loc} \)-to-\( \mr{L}^\infty_\mr{loc} \) estimate for $1 < p < \infty$ of a form originally due to Chou \& Kwong in \cite{chou2020general}.
Our proof has the bonus of clarifying the dependencies of the constant \( C \) below. 

\begin{thm}[Instantaneous \( \mr{L}^{p}_\mr{loc} \)-to-\( \mr{L}^\infty_\mr{loc} \)--estimate for $p > 1$; cf.~{\cite[Proposition 2.1]{chou2020general}}] \label{thm:intLpest}
Let \( u \in \Cinfloc(J_1 \times [0, T) \! ) \) be a Graphical Curve Shortening Flow starting from \( u_0 := u(0) \in \mr{L}^p(J_1) \) for some \( 1 < p < \infty \). 
Then for some universal  (that is, \emph{independent} of \( p \) and \( u \)) constant \( C < \infty \) there holds
\ms
| u(0, t) | \leq C \left(1 + \frac{\| u_0 \|_{\mr{L}^p(J_1)}^{p/(p - 1)}}{ t^{1/(p-1)}} + t \right)  \label{Lpuupbdd}
\mf
for all \( t \in (0, T) \). 
\end{thm}
 
By combining Theorem \ref{thm:intLpest} with Evan \& Spruck's estimate (\cite{evans1992motion}), one can deduce an instantaneous \( \mr{L}^\infty_\mr{loc} \)--estimate for the gradient of a GCSF starting from an \( \mr{L}^{p}_\mr{loc} \)--function, where $1 < p < \infty$.
Informally, the gradient estimates assert that \( \log(|\pd_x u(0, t)|) \) blows up like \(  t^{- (p +1)/(p - 1)} \) for times \( t \ll 1 \) and grows linearity for times \( t \gg 1 \).  \\

As another application of Theorem \ref{thm:intheightest}, in \S\ref{SEC:measgcsf} we construct solutions which start from \emph{non-atomic real–valued Radon measures} in the \emph{weak} sense. 
We denote the space of such measures by \( \mc{M}_*(\mb{R}) \), and refer to App.~\ref{app:meas} for further details, in particular the notion of weak convergence in the larger space \( \mc{M}(\mb{R}) \). 

\begin{thm}[Existence from $\mc{M}_*(\mb{R})$--initial data]\label{thm:measgcsf}
Let \( \nu  = u_0\msc{L}^1 + \nu_\mr{sing} \in \mc{M}_*(\mb{R}) \) be a non-atomic real-valued Radon measure, decomposed into its absolutely continuous \( \nu_{\mr{ac}} =: u_0 \msc{L}^1 \) and its singular \( \nu_\mr{sing} \) part; here \( \msc{L}^1 \) is the Lebesgue measure on \( \mb{R} \). 
Let \( \Omega := \mb{R} \setminus \supp(\nu_\mr{sing}) \) and suppose that \( u_0 \in \mr{L}^p_\mr{loc}(\Omega) \) for some \(1 \leq p < \infty \) (recall that the weakest case of \( p = 1 \) always holds by definition).
Then there exists a Graphical Curve Shortening Flow \( u \in \Cinfloc(\mb{R} \times (0, \infty) \! ) \) which observes the following modes of convergence as \( t \searrow 0 \):
\begin{align}
 u(t) \msc{L}^1 \rightharpoonup \nu \qquad &\text{ weakly in } \, \mc{M}(\mb{R}) \label{convM} \\
u(t) \to u_0 \qquad &\text{ strongly in } \, \mr{L}^p_\mr{loc}(\Omega).  \label{convL1loc}
\end{align}
\end{thm}

We also take the opportunity to comment on some of Chou \& Kwong's results from \cite{chou2020general}.
In the case of \( \Omega = \mb{R} \) in Theorem \ref{thm:measgcsf},
\begin{itemize}
\item if \( u_0 \in \mr{L}^p(\mb{R}) \), then the convergence \eqref{convL1loc} holds globally, i.e. 
\ms \label{convLp} 
u(t) \to u_0 \qquad \textit{ strongly in } \, \mr{L}^p(\mb{R}). 
\mf
\item if \( p > 1 \), then the solution is unique;
in fact, the mapping \( \mr{L}^p_\mr{loc}(\mb{R}) \ni u_0 \mapsto u \in \mr{C}^0_\mr{loc}([0, \infty); \mr{L}^p_\mr{loc}(\mb{R})\!) \) is well-defined and  continuous---see Theorem \ref{thm:Lploccont}.
\end{itemize}

In the final section (\S\ref{SEC:correspond}) of the paper, we define a class of GCSFs which arises naturally from our existence theory, and use it to further a correspondence between GCSFs and measures initiated by Chou \& Kwong in \cite{chou2020general}. 
The discussion is analogous to that from the two--dimensional RF theory of Topping \& Yin and Peachey \& Topping in \cite{topping2021smoothing, topping2023uniqueness} and \cite{peachey2024twodimensional} respectively. 
Whether our class captures all GCSFs is unknown; 
if so, then whether Chou \& Kwong's correspondence is one-to-one hinges on the natural question of uniqueness raised by our existence result Theorem \ref{thm:measgcsf}.

\subsection{Acknowledgements}
The author would like to thank his supervisor Prof.~Peter M.~Topping for the countless and invaluable discussions and suggestions which lead to this paper, and to thank Drs.~Ben S.~Lambert and Huy T.~Nguyễn for a number of helpful comments. 
The author also thanks the anonymous referee at Discrete and Continuous Dynamical Systems for their corrections and suggestions, which greatly improved the exposition here. 

The author was supported by  the Warwick Mathematics Institute Centre for Doctoral Training, and gratefully acknowledges funding from the University of Warwick and the UK Engineering and Physical Sciences Research Council (Grant number: \href{https://gtr.ukri.org/projects?ref=EP/R513374/1}{EP/R513374/1}).

The author has applied a \href{https://creativecommons.org/licenses/by/4.0/deed.en}{Creative Commons ``Attribution 4.0 International''} ({\href{https://creativecommons.org/licenses/by/4.0/deed.en}{CC BY 4.0}) licence to any author-accepted manuscript version arising from this submission.

\section{A local Harnack quantity}\label{SEC:harnquan}

In this section, we shall construct the local \emph{Harnack quantity} and prove the corresponding \emph{Harnack inequality}. 
As mentioned in \S\ref{sec:results}, our local Harnack quantity is only defined for the class of \emph{local GCSFs}, which we define below.
But first, we require a lemma which confirms that our definition is not ill.

\begin{lem}\label{lem:locgcsf}
Let \( u_0 \in \mr{C}^{0,1}(J_1 ; \mathbb{R}_{\geq 0}) \) be a positive Lipschitz function.
We construct the Lipschitz curve \( \gamma_0 : \mb{R} \mapsto \mb{R}^2 \) as illustrated in Fig.~\ref{fig:gam0fromu0}: 
for \( x \in J_1 \), set \( \gamma_0(x) = (x, u_0(x)\!) \); for \( x = 1 + y \in [1, \infty) \), set \( \gamma_0(x) = (1, u_0(1) + y) \); and for \( x = - 1 - y \in (-\infty, -1] \), set \( \gamma_0(x) = (-1, u_0(-1) + y) \).
Then {the} Chou--Zhu flow \( [0, \infty) \ni t \mapsto \gamma(t) \subset \mb{R}^2 \) of \( \gamma_0 =: \gamma(0) \) admits a graphical representation for some \( u \in \Cinfloc(J_1 \times (0, \infty)\!) \) with \( u(t) \to u_0  \in \Czloc(J_1) \) as \( t \searrow 0 \).
Moreover, \( u(y, t), \pm \pd_x u(y, t) \to \infty \) locally uniformly in \(t \in (0, \infty) \) as \( y \to \pm 1 \).
\end{lem}

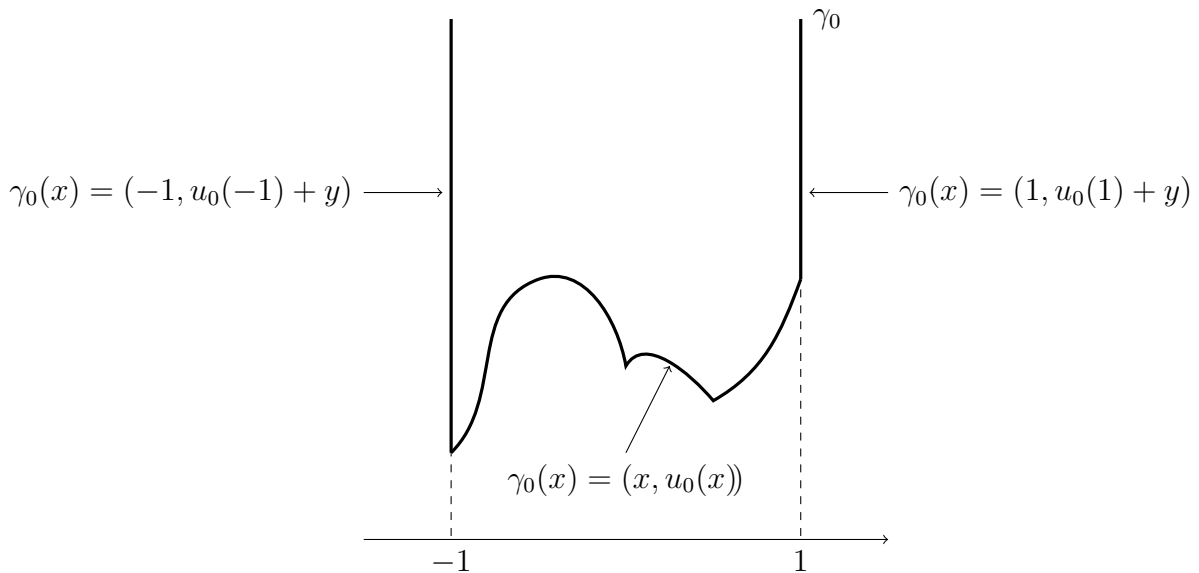
\begin{figure}[h]
\centering
\begin{tikzpicture}[scale=2.3, CENTRE]
\draw [->] (-1.5, 0) -- (1.5, 0);
\draw [dashed] (-1,3) -- (-1, 0) node[below]{${-1}$};
\draw [dashed] (1,3) -- (1, 0) node[below]{$1$};
\draw [very thick] (-1, 0.5) to[out=45, in=200] (-0.5, 1.5) to[out=200+180, in=100] (0, 1)  to[out=60, in=130] (0.5, 0.8)  to[out=30, in=250] (1, 1.5);
\draw [very thick] (-1,3) to (-1, 0.5) ;
\draw [very thick] (1,3) node[right]{$\gamma_0$} to (1, 1.5) ;
\draw [->] (0,.5) node[below]{$\gamma_0(x) = (x, u_0(x)\!)$} to (0.25, 1);
\draw [->] (1.5,2) node[right]{$\gamma_0(x) = (1, u_0(1) + y)$} to (1.05, 2);
\draw [->] (-1.5,2) node[left]{$\gamma_0(x) = (-1, u_0(-1) + y)$} to (-1.05, 2);
\end{tikzpicture}
\caption{Construction of \( \gamma_0 : \mb{R} \mapsto \mb{R}^2  \) from a given \( u_0 \in \mr{C}^{0,1}(J_1; \mb{R}_{\geq 0}) \).}
\label{fig:gam0fromu0}
\end{figure}

\begin{proof}
Recall from App.~\ref{app:avoidinter} that Angenent's Intersection Principle (Theorem \ref{thm:inter}) applies to \( \gamma 
\).
Restrict to the case of \( \gamma : \mb{R} \times [0, T] \mapsto \mb{R}^2 \) for an arbitrary \( T < \infty \).
\begin{enumerate}[1.]
\item \label{num:circ} First, we show that each \( \gamma(t) \) for \( t \in [0, T] \) is asymptotic to the vertical lines \( V_{\pm 1} := \set{ \pm 1 } \times \mb{R} \subset \mb{R}^2 \). 
Applying the Avoidance Principle (Theorem \ref{thm:avoidii})  to \( \gamma \) and to huge shrinking circles placed either to the left of \( V_{-1} \) or to the right of \( V_1 \) shows that \( \gamma(t) \) must live within the strip \( \overline{J_1} \times \mb{R} \subset \mb{R}^2 \).
Applying the Avoidance Principle again with huge circles placed below the \( x \)--axis shows that \( \gamma(t) \) must actually live within the half-strip \( \overline{J_1} \times \mb{R}_{\geq 0} \).
The Strong Maximum Principle implies that \( \gamma(t) \) cannot touch the boundary of this half-strip. 

The use of circles as barriers is a well-known technique; within the half-strip \( J_1 \times \mb{R}_{>0} \), we use different barriers: \emph{Angenent oval solutions}.
These are the closed ancient solutions \( (-\infty, 0) \ni s \mapsto \mfr{o}(s) \subset \mb{R}^2 \) given by vertical translations of the curves  
\ms
\mfr{o}(s) := \left\{ ( x, y ) \in J_1 \times \mb{R} \, \middle| \, \cosh\left( \frac{\pi y}{2} \right) = e^{-\pi^2 s / 4} \cos\left( \frac{\pi x}{2} \right) \right\} \subset \mb{R}^2.
\mf
Let \( L > \| u_0 \|_{\mr{L}^\infty(J_1)} \);  by translating  \( \mfr{o}(s) \) upwards to touch the line \( H_L := \mb{R} \times \{ L \} \subset \mb{R}^2  \) from above, and by considering  the flow \( [0, |s|) \ni t \mapsto \mfr{o}(t + s) \), the Avoidance Principle ensures that the flow of \( \gamma_0 \) must lie outside the flow of the Angenent oval solution---see Fig.~\ref{fig:asymptotics}---for any \( s < 0 \).   
In particular, a direct calculation (see e.g.~\cite[\S 2.5.2]{sobphd}) shows that as \( s \searrow - \infty \), the Angenent oval solutions approach the vertical translation of the \emph{Grim Reaper solution}
\( [0 , \infty) \ni t \mapsto \mfr{g}(t) \subset \mb{R}^2 \), given by
\ms
\mfr{g}(t) := \left\{ (x, y) \in J_1 \times \mb{R} \, \middle| \, e^{-\pi y/2} = e^{- \pi^2 t/4} \cos\left( \frac{\pi x}{2} \right) \right\} \subset \mb{R}^2,\label{grimreaper}
\mf
which initially touches \( H_L \) from above. 
 Since the Grim Reaper solution is asymptotic to the vertical lines \( V_{\pm 1} \), the  curve \( \gamma(t) \) must be too.
 
\item \label{num:line} Next, we show that the curve \( \gamma(t) \) is given by a graph in two senses. 
We first apply Angenent's Intersection Principle (Theorem \ref{thm:inter}) to \( \gamma \) and to the static flow of the vertical lines \( V_r := \set{r} \times \mb{R}  \subset \mb{R}^2 \) for \( r \in J_1 \); see Fig.~\ref{fig:asymptotics}. 
 There is initially a single intersection coming from the graphical representation \( \gamma_0= \graph(u_0) \), and for strictly positive times the curve \( \gamma(t) \) must cross \( V_r \) at least once; thus Angenent's Intersection Principle  ensures that the curve \( \gamma(t) \) cross each \( V_r \) exactly once. 
 Hence \( \gamma(t) =: \graph(u(t)\!) \) is a graph of a function \( u(t) : J_1 \mapsto \mb{R}_{>0} \).
 
 We second apply Angenent's Intersection Principle  to \( \gamma \) and to the horizontal lines \( H_{r'} := \mb{R} \times \set{r'} \subset \mb{R}^2 \); see Fig.~\ref{fig:asymptotics}.
 Let \( \underline{r}' \gg 1 \) be such that \( u(0, t) < \underline{r}' \) for all \( t \in  [0, T] \).
 Then, similar to before, we have 
  for  all \( {r'} \geq  \underline{r}' \) that each \( \gamma(t) \) for \( t \in [0, T] \) intersects \( H_{r'} \) exactly twice, and therefore that
  the ends of \( \gamma(t) \) are graphical over \( \set{\pm 1} \times [\underline{r}', \infty) \subset V_{\pm 1} \).
 
\item \label{num:reg} Finally, we establish the convergence properties from regularity theory. 
We know from \ref{num:line}. that, say, for  \(t \in [0, T] \) the right-most end of \( \gamma(t) \) can be written as a graph of \( \widetilde{u}(t) : [\underline{r}', \infty) \mapsto \mb{R} \) over \( \set{1} \times [\underline{r}', \infty) \subset V_1 \), which starts off as \( \widetilde{u}(0) = 0 \) but which instantly lifts off from zero (i.e.~\( \widetilde u(t) > 0 \) for all \( t \in (0, T]\)).
From \ref{num:circ}., we know that the graphing function \( \widetilde{u}(y, t) \) decays rapidly, say \( |\widetilde{u}(y, t)| \leq \Psi(y) \to 0 \), as \(y \nearrow \infty \), since \( \gamma(t) \) is squeezed between the vertical line \( V_1 \) and an upward-translated Grim Reaper \eqref{grimreaper}.

Ecker \& Huisken's gradient estimate \cite[2.1 Theorem]{ecker1991interior} implies that \( \pd_x \widetilde{u}(y, t) \) remains uniformly bounded for, say, \( (y,t) \in [\underline{r}'', \infty) \times [0, T]  \) for some \( \underline{r}'' \gg \underline{r}'  \). 
Standard parabolic regularity theory (see e.g.~\cite{ladyzenskaja1988linear, krylov1996lectures}) then ensures that the gradient \( \pd_x \widetilde{u}(y, t) \) of the graphing function must also decay to zero as \( y \nearrow \infty \) at the same rate as \( \widetilde{u}(y, t) \): \( |\pd_x \widetilde{u}(y, t)| \leq C \Psi(y) \to 0 \).
Written as a graph over \( J_1 \times \{ 0 \} \), this implies that both \( u(y, t) \) and \( \pd_x u(y, t) \) blow up to infinity as \( y \nearrow 1 \), locally uniformly in \( t \in (0, T] \). 
The argument for the left-most end of \( \gamma(t) \) is similar.
\qedhere
 \end{enumerate} 
\end{proof}

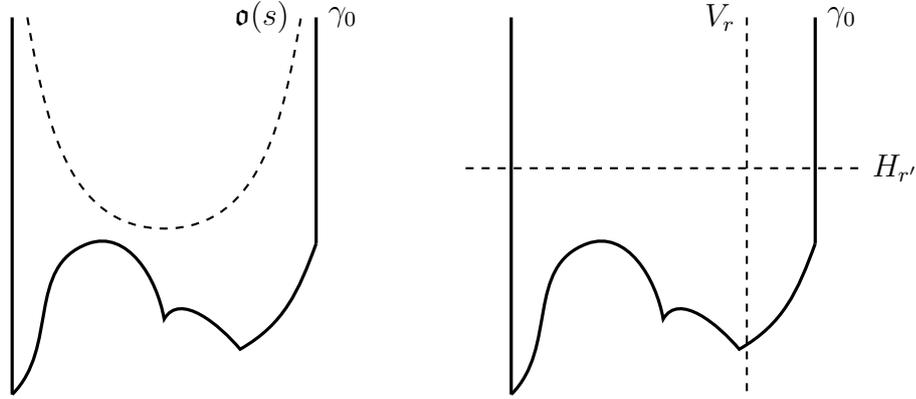
\begin{figure}[h]
\centering
\begin{tikzpicture}[scale=2] 
\draw [very thick] (-1, 0.5) to[out=45, in=200] (-0.5, 1.5) to[out=200+180, in=100] (0, 1)  to[out=60, in=130] (0.5, 0.8)  to[out=30, in=250] (1, 1.5);
\draw [very thick] (-1,3) to (-1, 0.5) ;
\draw [very thick] (1,3) node[right]{$\gamma_0$} to (1, 1.5) ;
\draw [thick, dashed] (-0.9, 3) to[out=280, in=180] (0,1.6) to[out=0, in=260] (0.9, 3) node[left]{$\mfr{o}(s)$};
\path (-1.7, 2);
\end{tikzpicture}
\hspace{1cm}
\begin{tikzpicture}[scale=2] 
\draw [very thick] (-1, 0.5) to[out=45, in=200] (-0.5, 1.5) to[out=200+180, in=100] (0, 1)  to[out=60, in=130] (0.5, 0.8)  to[out=30, in=250] (1, 1.5);
\draw [very thick] (-1,3) to (-1, 0.5) ;
\draw [very thick] (1,3) node[right]{$\gamma_0$} to (1, 1.5) ;
\draw [thick, dashed] (0.55,3) node[left]{$V_r$} to (0.55, .5) ;
\draw [thick, dashed] (-1.3, 2) to (1.3, 2) node[right]{$H_{r'}$};
\end{tikzpicture}
\caption{Left: Applying the Avoidance Principle (Theorem \ref{thm:avoidii}) to the solid \( \gamma \) and to a dashed Angenent oval solution \( \mfr{o} \) in \ref{num:circ}. \\ Right: Applying Angenent's Intersection Principle (Theorem \ref{thm:inter})  to the solid \( \gamma \) and to a dashed vertical/horizontal line \( V_r \)/\( H_{r'} \) in \ref{num:line}.}
\label{fig:asymptotics}
\end{figure}

\begin{rem}[$\mr{L}^1$--dominating function]\label{rem:lem2.1}
The proof of Lemma \ref{lem:locgcsf}, in particular \ref{num:circ}., shows the following: 
Let \( L > \| u_0 \|_{\mr{L}^\infty(J_1)} \). 
Then the restricted solution \( u \in \Cinfloc(J_1 \times (0, T]) \) from Lemma \ref{lem:locgcsf} is dominated by the Grim Reaper solution \eqref{grimreaper} translated upwards by \( L \); that is,
\ms
u(x, t) \leq L + \frac{\pi t}{2} - \frac{2}{\pi} \log\left[ \cos\left( \frac{\pi x}{2} \right) \right] \leq L + \frac{\pi T}{2} - \frac{2}{\pi}  \log\left[ \cos\left( \frac{\pi x}{2} \right) \right] \label{grimreapup}
\mf
for all \( (x, t) \in J_1 \times (0, T] \). In particular, the right-hand side of \eqref{grimreapup} lives in \( \mr{L}^1(J_1) \).

Note that if one wished to apply the Avoidance Principle (Theorem \ref{thm:avoidii}) directly to $\gamma$ and the Grim reaper solution, both of which are non-compact, then one would have first needed to establish the asymptotic behaviour of $\gamma$. 
Working with families of Angenent oval solutions, rather than directly with the Grim reaper solution, circumvents this technicality.
\end{rem}

\begin{rem} For the purposes of \ref{num:reg}. in Lemma \ref{lem:locgcsf}, it was \emph{not} important to clarify how the constant \( C \) depended on the solution \( \widetilde{u} \), as we only required the \emph{qualitative} statement that \( \pd_x \widetilde{u}(y,t) \to 0 \) as \( y \nearrow \infty \) locally uniformly in \( t \in (0, \infty) \).
 In the proof of Theorem \ref{thm:measgcsf}, it will be paramount to obtain \emph{quantitative} estimates that depend on only weak information about the solution.
 Therefore, in \S\ref{sec:construction} we will carry out a procedure similar to above, but with a little more care.
\end{rem}

Lemma \ref{lem:locgcsf} gives rise to the class of \emph{local GCSFs} on which our local Harnack quantity is defined.

\begin{defn}[Local GCSF] \label{defn:locuflow}
Let \( u_0 \in \mr{C}^{0,1}( J_1 ; \mb{R}_{\geq 0} ) \) be a  positive Lipschitz function on \( J_1 \).
The \emph{local GCSF starting from \( u_0 \)}  is the solution \( u \in \Cinfloc(J_1 \times (0, \infty) \! ) \) 
to the GCSF constructed in Lemma \ref{lem:locgcsf}.
\end{defn}
Note that a local GCSF \( u \) extends to a function in \( \Czloc(J_1 \times [0, \infty)\!) \), and is, by definition, positive. \\

We now define the local \emph{Harnack quantity}, which is adapted from \cite[Lemma 2.1]{sobtopdelayed}.

\begin{defn}[{Local Harnack quantity; cf.~\cite[(2.6)]{sobtopdelayed}}]\label{defn:locharnquan}
Let \( u \in \Cinfloc( J_1 \times (0, \infty) \! )  \) be a local GCSF. 
Define the \emph{area} and \emph{angle\footnote{%
The quantity \( \phi(y, t) \) is the angle that the tangent to the curve \( \gamma(t):= \graph(u(t)\!) \) at \( (y, u(y, t)\!) \) makes with the downward vertical direction.%
}} \emph{functions} \( \mc{A} : J_1 \times [0, \infty) \mapsto [0, \infty] \) and \( \phi : J_1 \times (0, \infty) \mapsto (0, \pi) \) by 
\begin{align}
\mc{A}(y, t) &:= \int_{x \in (-1, y)} u(x, t) \, \mr{d}x  \label{area}\\
\intertext{and}
\phi(y, t) &:= \tan^{-1}[\pd_x u(y, t)] + \frac{\pi}{2} \label{angle}
\end{align}
respectively (see Fig.~\ref{fig:areaangle}).
Define the local \emph{Harnack quantity} \( \mc{H} : J_1 \times (0, \infty) \mapsto (-\infty, \infty] \) by
\ms
\mc{H}(y, t) := \mc{A}(y, t) - 2t \,  \phi(y, t). \label{harnack}
\mf
\end{defn}

\begin{figure}[h]
\centering
\begin{tikzpicture}[scale=2.3, CENTRE]
\fill [gray!30] (-0.95, 3) to[out=-89, in=110] (-0.5, 0.5) to[out=110+180, in=180] (0, 0.7) to[out=10, in=120] (0.5, 1) to (0.5,0) to (-1, 0) to (-1, 3) to (-0.95, 3);
\draw [->] (-1.5, 0) -- (1.5, 0);
\draw [dashed] (-1,3) -- (-1, 0) node[below]{${-1}$};
\draw [dashed] (1,3) -- (1, 0) node[below]{$1$};
\draw [very thick] (-0.95, 3) to[out=-89, in=110] (-0.5, 0.5) to[out=110+180, in=190] (0, 0.7) to[out=10, in=120] (0.5, 1) to[out=120+180, in=-91] (0.95, 3) node[left]{$\gamma(t)$};
\draw [] (.5, 1) -- (.5, 0) node[below]{$y$};
\draw [->, thick] (.5,1) -- (0.7, 0.7) node[below]{$\;\;\,\phi(y, t)$};
\draw [->] (-1.3, 1) node[left]{$\mc{A}(y, t)$}to (-.8, 1);
\draw[samples=40,domain=-90:-58] plot ({cos(\x)/5 + .5}, {sin(\x)/5+1});
\end{tikzpicture}
\caption{Illustration of \( \mc{A}(y, t) \) and \( \phi(y, t) \), where $\gamma(t) := \graph(u(t)\!)$.}
\label{fig:areaangle}
\end{figure}

 The content of the following proposition is to confirm that \( \mc{H} \) is smooth on \( J_1 \times (0, \infty) \), and  extends continuously to \( \overline{J_1} \times [0, \infty) \)  with well-controlled boundary behaviour:

\begin{prop}[{Regularity, initial condition and boundary values for \( \mc{H} \); cf.~\cite[Lemma 2.1]{sobtopdelayed}}]\label{lem:harninitbound}
Let \( \mc{H} : J_1 \times (0, \infty) \mapsto (-\infty, \infty] \) be the local Harnack quantity defined for the local GCSF \(  u \in \Cinfloc( J_1 \times (0, \infty) \! )  \) starting from \( u_0  \in \mr{C}^{0,1}(J_1) \).
Then \( \mc{H} \) extends to a function in \( \Cinfloc(J_1 \times (0 , \infty)\!) \cap \Czloc(\overline{J_1} \times [0, \infty) \! ) \), and satisfies the initial condition
\ms
\mc{H}(x, 0) = \mc{A}(x, 0) \geq 0 \label{init}
\mf
for all \( x \in J_1 \) and the boundary values
\ms
\begin{aligned}
\mc{H}(-1, t) &= 0  \geq - \pi t \\
 \mc{H}(1, t) &= \overline{\mc A} - \pi t := \| u_0 \|_{\mr{L}^1(J_1)} - \pi t\geq - \pi t 
\end{aligned}
\label{bound}
\mf
for all \( t \in [0, \infty) \); here \( \mc{A} : J_1 \times [0, \infty) \mapsto [0, \infty] \) is the area function of \( u \).
\end{prop}

\begin{proof}
It suffices to prove the regularity, extension and conditions when \( t \in [0, T] \), for an arbitrary \( T < \infty \). 

First, we show that \( \mc{H} \) is finite and extends to a continuous function on \( \overline{J_1} \times [0, T] \).
By Remark \ref{rem:lem2.1}, we know that the functions \( u(t) : J_1 \mapsto \mb{R}_{\geq 0} \) for \( t \in [0, T] \) are dominated by the same \( \mr{L}^1(J_1) \)--function. 
Hence the Dominated Convergence Theorem implies that the area function \( \mc{A}  : J_1 \times [0, T] \mapsto [0, \infty] \) is finite everywhere and extends continuously to  \( \overline{J_1} \times [0, T] \), taking the boundary values
\ms
\mc{A}(-1, t) =  0 \qquad \text{ and } \qquad \mc{A}(1, t) = \| u(t) \|_{\mr{L}^1(J_1)}.
\mf
The boundedness of \( \phi(y, t) \) and the locally uniform in \( t \in (0, T] \)--convergence of \( \pd_x u(y, t) \to \pm \infty \) as \( y \to \pm 1 \) given by Lemma \ref{lem:locgcsf} similarly imply that the function \( (x, t) \mapsto t \,\phi(x, t) : J_1 \times (0, T] \mapsto \mb{R} \) extends continuously to \( \overline{J_1} \times [0, T] \).
Thus so does \( \mc{H} : J_1 \times (0, T] \mapsto (-\infty, \infty]  \).

Next, we show that \( \mc{H} \) restricted to the parabolic boundary of \( J_1 \times (0, T] \), i.e.~to ($J_1 \times \set{0}) \cup (\set{\pm 1} \times [0, T])$, is given by \eqref{init}--\eqref{bound}.
The initial condition \eqref{init} and the first boundary condition of \eqref{bound} are immediate, since \( \phi(y, t) \to 0 \) as \( y \searrow -1 \);
we verify the second condition of \eqref{bound}.
Denote by \( v(y, t) := \pd_x u(y, t) \) the gradient of \( u(t) \) at \( y \).
Let \( (a, b) \Subset J_1 \) and consider 
\ms\notag
A_{a,b}(t) := \int_{y \in (a,b)} u(y, t) \eds y. 
\mf
Differentiating under the integral sign at time \( s \), we have that
\ms
\begin{aligned}
\pd_t A_{a,b}(s) &= \int_{y \in (a,b)} \pd_t u(y, s) \eds y \\
&= \int_{y \in (a,b)} \pd_x[ \tan^{-1}(v) ](y, s) \eds y \\
&= [\tan^{-1}(v)](b,s) - [\tan^{-1}(v)](a,s) \\
&= \phi(b, s) - \phi(a, s),
\end{aligned}
\mf
and so integrating up, we have that
\ms
A_{a,b}(t) = A_{a,b}(0) + \int_{s \in (0, t)} [\phi(b, s) - \phi(a,s)] \eds s.
\mf
The Dominated Convergence Theorem implies that both 
\ms A_{a,b}(s) \to \mc{A}(b, s) \qquad \text{ and } \qquad \int_{s \in (0, t)} [\phi(b, s) - \phi(a, s)] \eds s \to \int_{s \in (0, t)} \phi(b, s) \eds s
\mf as \( a \searrow -1 \), the former for each \( s \in [0, t] \), and hence that
\ms
\mc{A}(b, t) = \mc{A}(b, 0)  + \int_{s \in (0, t)} \phi(b, s) \eds s. \label{Aaltform}
\mf
Taking \( b \nearrow 1 \), we have that 
\ms
\mc{A}(b, t) \to \mc{A}(1, t) = \| u(t)\|_{\mr{L}^1(J_1)} = \overline{\mc A} + \pi t 
  \label{Afullint}
\mf 
and, for each \( s \in (0, t] \), that 
\ms
\phi(b, s) \to \pi. \label{phiconv}
\mf 
We shall use \eqref{Afullint} again below; for now, \eqref{Afullint}--\eqref{phiconv}  implies the second boundary condition of \eqref{bound}.

Finally, we show that \( \mc{H} \) is smooth on \( J_1 \times (0, T) \).
It suffices to show that both \( \pd_t \mc{A} \) and \( \pd_x \mc{A} \) are smooth; but this is clear, since by the Dominated Convergence Theorem one can differentiate \eqref{Aaltform} and \eqref{area} to get that
\begin{align}
\pd_t \mc{A} &= \phi  \label{At} \\
\intertext{and}
\pd_x \mc{A} &= u 
\end{align}
respectively, which are both clearly smooth.
\end{proof}

Armed with the regularity, initial condition and boundary values of \( \mc{H} \) provided by Proposition \ref{lem:harninitbound}, we are in the position to prove the key \emph{Harnack inequality}.

\begin{thm}[{Harnack inequality; cf.~\cite[(2.11)]{sobtopdelayed}}]\label{thm:Harnquanlowbdd}
Let \( \mc{H} \in \Cinfloc(J_1 \times (0, \infty)\!) \cap \Czloc(\overline{J_1} \times [0, \infty) \! ) \) be the local Harnack quantity defined for the local GCSF \(  u \in \Cinfloc( J_1 \times (0, \infty) \! ) \). 
Then \( \mc{H} \) satisfies the {\normalfont Harnack inequality}
\ms
 \mc{H}(x, t) \geq -\pi t \label{Harnquanlowbdd}
\mf
for all \( (x, t) \in J_1 \times [0, \infty) \).
\end{thm}

Our justification for calling $\mathcal{H}$ a ``Harnack quantity'', and consequently \eqref{Harnquanlowbdd} a ``Harnack inequality'', both for the $\mathcal{H}$ here and in \cite{sobtopdelayed}, is in its links to other Harnack quantities for other geometric flows---see, for example, Remark \ref{rem:ty17}.

For a concrete link between Hamilton's Harnack Inequality \cite[Main Theorem A]{hamilton1993harnack} for the MCF and an adapted version of our Harnack inequality \cite[(2.11)]{sobtopdelayed}/\eqref{Harnquanlowbdd}, the interested reader is directed to \cite[\S5]{sobtopharnack}, where Topping \& the author show, in the setting of the CSF of convex curves, that the two Harnack inequalities imply a common pointwise curvature estimate.  

\begin{proof}[Proof of Theorem \ref{thm:Harnquanlowbdd}]
We first establish a differential inequality for \( \mc{H} \).
Again, denote by \( v(y, t) = \pd_x u(y, t) \) the gradient of \( u(t) \) at \( y \). 
We already know that \( \pd_t \mc{A} = \phi \) from \eqref{At}.
For the other derivatives, we compute directly from the definitions \eqref{area}--\eqref{angle} that
\begin{align}
\pd^2_x \mc{A} &= v, \label{Axx}\\
\pd_t \phi &= \pd_t [  \tan^{-1}(v)] = \frac{1}{1 + v^2} \pd_t v \label{pt}\\
\intertext{and}
\pd_x^2 \phi &= \pd_x^2 [\tan^{-1}(v) ] = \pd_t v. \label{pxx}
\end{align}
Thus, combining \eqref{At} and \eqref{Axx}--\eqref{pxx} we have that
\begin{align} 
\pd_t \mc{H} &= \pd_t \mc{A} - 2 \, \phi - 2t \, \pd_t \phi = - \phi - \frac{2 t}{1 + v^2} \pd_t v \\
\intertext{and}
 \pd^2_x \mc{H} &=  \pd^2_x \mc{A} - 2t \, \pd^2_x \phi = v - 2 t \, \pd_t v.  
 \end{align}
 Therefore
 \begin{align}
\pd_t \mc{H} - \frac{1}{1 + v^2} \pd_x^2 \mc{H} &= - \phi - \frac{v}{1 + v^2} = - \tan^{-1}(v) - \frac{v}{1 + v^2} - \frac{\pi}{2} \geq - \pi, \label{pdiforharn}
\end{align}
using that \( |\!\tan^{-1}(v) + v/(1 + v^2) | \leq \pi / 2 \) for all \( v \in \mb{R} \), as in \cite[Lemma 2.1]{sobtopdelayed}.

A consequence of Proposition \ref{lem:harninitbound} and \eqref{pdiforharn} is that the quantity
\ms
\widetilde{\mc{H}}(x, t) = \mc{H}(x, t) + \pi t
\mf
is positive on the parabolic boundary of \( J_1 \times (0, \infty) \) and satisfies the differential inequality 
\ms
\pd_t \widetilde{\mc{H}} - \frac{1}{1 + v^2} \pd_x^2 \widetilde{\mc H} \geq 0
\mf
within \( J_1 \times (0, \infty) \).
From the Maximum Principle, we deduce the required \( \widetilde{\mc H}(x, t) \geq 0 \). 
\end{proof}

It is clear that our local Harnack quantity and our Harnack inequality (Definition \ref{defn:locharnquan} and Theorem \ref{thm:Harnquanlowbdd} respectively) are strongly linked to the Harnack quantity and inequality from \cite[Lemma 2.1]{sobtopdelayed}.  
In fact, there is also a link to the Harnack quantity of Topping \& Yin (see \cite{topping2017sharp}) for the two--dimensional RF.

  \begin{rem}[{Link to the Harnack quantity in \cite[\S2.2]{topping2017sharp}}]\label{rem:ty17}
 Recall that the conformal factor \( u \) of a two--dimensional RF in local isothermal coordinates \( (x, y) \in \mb{B}^2 \) satisfies the \emph{Logarithmic Fast Diffusion Equation}
 \ms
\tag{LFDE} \label{LFDE} \pd_t u = \Delta \log(u) ,
\mf
where \( \Delta := \pd_x^2 + \pd_y^2 \).
Notice that \eqref{LFDE} is structurally similar to the equation satisfied by the spatial derivative \( v \) of a GCSF,
\ms\label{gcsf'}
 \pd_t v =  \pd_x^2 [ \tan^{-1}(v) ].
\mf
 In \cite[\S 2.2]{topping2017sharp}, Topping \& Yin looked for a potential \( \varphi \in \Cinfloc(\mb{B}^2 \times [0, \infty) \! ) \) for which
\ms
\Delta \varphi = u 
\mf
and
\ms
 \pd_t \varphi = \log ( \Delta \varphi ) + c 
\mf
 with \( c = 1 \).
 Our area function \( \mc{A}  \in \Cinfloc(J_1 \times (0, \infty) \! ) \)  satisfies a similar pair of conditions: 
\ms
 \pd_x^2 \mc{A} = v 
 \mf
from \eqref{Axx} and
\ms
\pd_t \mc{A} = \phi = \tan^{-1}(v) + \frac{\pi}{2} = \tan^{-1} ( \pd_x^2 \mc{A} ) + \frac{\pi}{2}
\mf
 from \eqref{At}.
 \end{rem}

\section{The delayed interior area-to-height estimate}\label{SEC:locgradest}

In this section, we shall use our local Harnack quantity to deduce our delayed interior area-to-height estimate (Theorem \ref{thm:intheightest}) for an arbitrary GCSF.
Our method is to reduce to the case of a local GCSF, and then use our local Harnack inequality (Theorem \ref{thm:Harnquanlowbdd}) to deduce the height bound. \\

We start by making two reductions. 
First, we  claim that it suffices to consider GCSFs which extend to the boundary.

\begin{claim}\label{claim:suffubdd}
If Theorem \ref{thm:intheightest} holds for GCSFs \( u \in \Cinfloc(J_1 \times [0, T) \! ) \) for which \( u(t) \in \mr{C}^{0,1}(J_1) \) for every \( t \in [0, T) \), then Theorem \ref{thm:intheightest} also holds in full generality.
\end{claim}

\begin{proof}
Let \( u \in \Cinfloc(J_1 \times [0, T) \! ) \) be a GCSF of the generality considered by Theorem \ref{thm:intheightest}.
For any \( 0 < \rho < 1 \), we consider the rescaled GCSF \( u^\rho \in \Cinfloc(J_1 \times [0, \rho^{-2} T)\!) \) given by
\ms
u^\rho(x, t) := \frac{1}{\rho} u(\rho x, \rho^2 t ), \label{urho}
\mf
which starts from \( u^\rho_0(x) := \rho^{-1} u_0(\rho x) \).
Clearly then \( u^\rho(t) \in \mr{C}^{0,1}(J_1) \) for every \( t \in [0, \rho^{-2} T) \) (it is, in fact, smooth up to the boundary), and we can readily compute that 
\ms
\| u_0^\rho \|_{\mr{L}^1(J_1)} = \int_{x \in J_1} u^\rho_0(x) \eds x = \frac{1}{\rho^2} \int_{x \in J_\rho} u_0(x) \eds x = \frac{1}{\rho^2} \| u_0 \|_{\mr{L}^1(J_\rho)} \label{urhoL1} 
\mf
{and}
\ms
\left| u^\rho\left(0, \frac{t}{\rho^2} \right) \right| = \frac{1}{\rho} | u(0,  t) |. \label{urhoLinf}
\mf
Let \( \delta > 0 \). 
Then by applying Theorem \ref{thm:intheightest} to the GCSF \( u^\rho \), we have from \eqref{uupbdd} and \eqref{urhoL1}--\eqref{urhoLinf} that
\ms
\frac{1}{\rho} | u(0, t) | \leq C  + \frac{1}{2\rho^2} \| u_0 \|_{\mr{L}^1(J_\rho)} + \frac{\pi t}{2 \rho^2}  \label{urhoupbdd}
\mf
for all 
\ms
\frac{t}{\rho^2} \geq (1 + \delta) \cdot \frac{ \|u_0\|_{\mr{L}^1(J_1)} }{\rho^2 \pi}  \geq (1 + \delta) \cdot \frac{ \|u_0\|_{\mr{L}^1(J_\rho)} }{\rho^2 \pi} , \label{urhotdel}
\mf
for some \( C < \infty \) depending on only \( \delta \). 
We recover the full case of Theorem \ref{thm:intheightest} by letting \( \rho \nearrow 1 \) in \eqref{urhoupbdd}.
\end{proof}

Second, we claim that it suffices to prove Theorem \ref{thm:intheightest} for local GCSFs.

\begin{claim} \label{claim:suffuloc}
If Theorem \ref{thm:intheightest} holds for local GCSFs, then Theorem \ref{thm:intheightest} also holds for GCSFs \( u \in \mr{C}^\infty_\mr{loc}(J_1 \times [0, T) \! ) \) for which \( u(t) \in \mr{C}^{0,1}(J_1) \) for every \( t \in [0, T) \).
\end{claim}

\begin{proof}
Let  \( u \in \mr{C}^\infty_\mr{loc}(J_1 \times [0, T) \! ) \) be a GCSF for which \( u(t) \in \mr{C}^{0,1}(J_1) \) for every \( t \in [0, T) \), with \( u_0 := u(0)  \).
We may suppose that \(  u_0 \neq 0 \) (in particular, that \(  \|u_0\|_{\mr{L}^1(J_1)} > 0 \)), for otherwise the result is an immediate consequence of the Avoidance Principle (Theorem \ref{thm:avoidii}).
For any \( \eta > 0 \), 
consider the positive Lipschitz function given by
\ms
\overline{u}_0(x) := |u_0(x)|  + \eta >  u_0(x) \label{ubareta}
\mf
Let \( \overline{u} \in \Cinfloc(J_1 \times (0, \infty)\!) \) be the local GCSF starting from \( \overline{u}_0 \). 
Since each \( u(t) \) is bounded on \( \overline{J_1} \), it follows from the Avoidance Principle (Theorem \ref{thm:avoidii}) and the initial ordering \eqref{ubareta} that the graphs of the local GCSF \( \overline{u}(t) \) and the original GCSF \( u(t) \) are disjoint for every \( t \in [0, T) \). 
In particular, we must have 
\ms
u(x, t) \leq \overline{u}(x, t)
\mf
for all \( (x, t) \in J_1 \times [0, T) \).
Applying Theorem \ref{thm:intheightest} to \( \overline{u} \) with \( \delta/2  \), while noting that
\ms
\| \overline{u}_0 \|_{\mr{L}^1(J_1)} = \| {u}_0 \|_{\mr{L}^1(J_1)} + 2\eta >  \| {u}_0 \|_{\mr{L}^1(J_1)} > 0 ,
\mf
we have, for any \( \delta > 0 \) and some \( C < \infty \) depending on only \( \delta \), that
\ms
u(0, t) \leq  \overline{u}(0, t) = |\overline{u}(0, t)| \leq C + \frac{1}{2}
(\| {u}_0 \|_{\mr{L}^1(J_1)} + 2\eta) 
+ \frac{\pi t}{2} 
\mf
for all
\ms
t \geq (1 + \delta) \cdot \frac{ \| u_0 \|_{\mr{L}^1(J_1)}}{\pi}  \geq \left( 1 + \frac{\delta}{2} \right) \cdot \frac{\| u_0 \|_{\mr{L}^1(J_1)} + 2 \eta}{\pi}, 
\mf
as least for \( \eta > 0 \) small enough. 
By taking \( \eta \searrow 0 \), we obtain the desired upper bound for \( u(0, t) \). 
Repeating the argument above with \( -u \) replacing \( u \) yields the desired lower bound.
\end{proof}

The upshot of Claims \ref{claim:suffubdd}--\ref{claim:suffuloc} is, that to prove Theorem \ref{thm:intheightest} in full generality, it suffices to prove the following version of the theorem:

\begin{thm}\label{thm:intheightesteven}
Let  \( u \in \Cinfloc(J_1 \times (0, \infty)\!) \) be a local GCSF starting from \( u_0 \in \mr{C}^{0,1}(J_1) \), and let
\ms
t_\star := \frac{\overline{\mc{A}}}{\pi} =  \frac{\| u_0 \|_{\mr{L}^1(J_1)}}{\pi} =  \frac{1}{\pi} \int_{x \in J_1} u_0(x) \eds x.
\mf
Then for all \( \delta > 0 \), there exists a constant \( C = C(\delta) < \infty \) such that the bound
\ms
|u( 0,  t )| \leq C  + \frac{\overline{\mc{A}}}{2} + \frac{\pi t}{2}  \label{uupbddeven}
\mf
hold for all \( t \geq (1+ \delta) \cdot t_\star \). 
\end{thm}

We prove this now, by appealing to our Harnack inequality (Theorem \ref{thm:Harnquanlowbdd}).

\begin{proof}[Proof of Theorem \ref{thm:intheightesteven}]
Fix \( \delta > 0 \).
Let \( \mc{H} \in \Cinfloc(J_1 \times (0, \infty) \! ) \) be the local Harnack quantity for \( u 
 \).
Then by the Harnack inequality (Theorem \ref{thm:Harnquanlowbdd}), we have that
\ms
\mc{H}(y, t) \geq - \pi t.
\mf
for all \( (y, t) \in J_1 \times (0, \infty) \). 
Unpacking the definition \eqref{harnack} of \( \mc{H} \), equivalently we find that 
\ms
\tan^{-1}[\pd_x u(y, t)]\leq \frac{\mc{A}(y, t)}{2 t}. \label{arctanbdd}
\mf

Fix \( t \geq (1 + \delta) \cdot t_\star = (1 + \delta) \cdot \overline{\mc A}/\pi \), and let \( c_t  \in J_1 \) be such that
\ms
\mc{A}(c_t, t) = \frac{1}{2} \mc{A}(1, t) = \frac{\overline{\mc A}}{2} + \frac{\pi t}{2},
\mf
using \eqref{Afullint}. 
By replacing \( u(x, t) \) with \( u(-x, t) \) if necessary, we may suppose that \( c_t \geq 0 \); consequently, for \( y \leq 0 \leq c_t \) we have that 
\ms\label{leftAbdd}
\mc{A}(y, t) \leq \mc{A}(c_t, t) = \frac{\overline{\mc A}}{2} + \frac{\pi t}{2}.
\mf 
Then, using \eqref{arctanbdd}--\eqref{leftAbdd}, we obtain that
\msa{}
\tan^{-1}[\pd_x u(y, t)] &\leq \frac{1}{2t} \left( \frac{\overline{ \mc A}}{2} + \frac{\pi t}{2} \right) = \frac{\overline{\mc A}}{4 t} + \frac{\pi}{4} \leq \frac{\pi}{4(1 + \delta)} + \frac{\pi}{4} < \frac{\pi}{2}.  
\mfa
Hence
\ms
\pd_x u(y, t) \leq 2 C = 2 C(\delta) := \tan\left[ \frac{\pi}{4(1+\delta)} + \frac{\pi}{4} \right], \label{pdxuupbdd}
\mf
again, when \( y \leq 0 \) and \( t \geq (1 + \delta) \cdot t_\star \).

To deduce the height bound, we use the upper gradient bound \eqref{pdxuupbdd} to argue that the region trapped under \( \gamma(t) := \graph(u(t)\!) \) must contain a trapezium, and then use the upper area bound \eqref{leftAbdd} to control the area of the trapezium---see Fig.~\ref{fig:lowerareabdd}.
Let \( t \geq (1 + \delta) \cdot t_\star \) be as before, and set \( h := u(0, t) \).
The upper gradient bound \eqref{pdxuupbdd} implies that the line of gradient \( 2 C \) through \( (0, h) = (0, u(0,t)\!) \in \mb{R}^2 \) must lie below \( \gamma(t) \) to the left of the origin; in other words,
\ms
u(x, t) \geq h + 2Cx  \label{lowubdd}
\mf
for all \( -1 < x < 0 \).
Integrating \eqref{lowubdd} over \( x \in (-1, 0) \) and using the area bound \eqref{leftAbdd}, we get that
\ms
h - C = \int_{x \in (-1,0)} [h + 2 C x] \eds x\leq \int_{x \in (-1,0)} u(x, t) \eds x = \mc{A}(0, t) \leq  \frac{\overline{\mc A}}{2} + \frac{ \pi t}{2},
\mf
which rearranges to
\ms
| u(0, t) | = u(0, t) = h \leq   C + \frac{\overline{\mc A}}{2} + \frac{ \pi t}{2}.  \qedhere
\mf
\end{proof}

\noindent {\it The proof of Theorem \ref{thm:intheightest} is now complete.}

\begin{rem}[Estimate for \( C = C(\delta) \)\!]\label{rem:asympCd}
By rearranging \eqref{pdxuupbdd}, we get that
\msa{}
C(\delta) = \frac{1}{2} \tan\left[ \frac{\pi}{2} - \frac{\pi \delta}{4(1 + \delta)} \right] = \frac{1}{2} \cot\left[ \frac{\pi \delta}{4(1+\delta)} \right] \leq \frac{1}{2} \csc\left[ \frac{\pi \delta}{4(1+\delta)} \right].
\mfa
Together with the lower bound \( \sin(x) \geq 2x/\pi \) for \( 0 < x < \pi/2 \), we deduce that
\ms
C(\delta) \leq \frac{1 + \delta}{\delta}. 
\mf 
Therefore, by choosing $\delta_t > 0$ so that 
\ms
t =: (1 + \delta_t) \cdot t_\star \iff \delta_t := \frac{t - t_\star}{t_\star}, 
\mf
we obtain the bound 
\ms
C(\delta_t) \leq \frac{t}{t - t_\star}.
\mf
\end{rem}

\begin{rem}[Delayed global area-to-height estimate]\label{rem:locproofofglob}
We highlight that  Remark \ref{rem:asympCd} and the proof of Theorem \ref{thm:intheightest} given here can be adapted and applied to the \emph{global} Harnack quantity in \cite[Lemma 2.1]{sobtopdelayed} to give an alternative proof of the following version of Theorem \ref{thm:globheightest}: \emph{For all positive GCSFs \( u \in \Cinfloc(\mb{R} \times [0, \infty); \mb{R}_{\geq 0}) \) with \( u_0 := u(0) \in \mr{L}^1(\mb{R}; \mb{R}_{\geq 0}) \), and for all times \( t > t_\star := \| u_0 \|_{\mr{L}^1(\mb{R})} / \pi \), there holds the height bound
\ms
\| u(t) \|_{\mr{L}^\infty(\mb{R})} \leq \sqrt{t} \cdot \sqrt{ \frac{2 \| u_0 \|_{\mr{L}^1(\mb{R})} }{t - t_\star} }.
\mf}
\end{rem}

\begin{figure}[h]
\centering
\begin{tikzpicture}[scale=2.3, CENTRE]
\fill [gray!30] (-0.95, 3) to[out=-89, in=110] (-0.5, 1.1) to[out=110+180, in=180] (0, 1.2) to[out=0, in=-100] (0.3, 1.4) to (.3,0) to (-1,0) to (-1, 3) to (-0.95, 3);
\draw (.3,1.4) to (0.3,0) node[below]{$c_t$};
\draw [pattern=crosshatch, pattern color=black!60]  (-1,0) to (0, 0) to (0, 1.2) to (-1, 0.1) to (-1,0);
\draw [->] (-1.5, 0) -- (1.5, 0);
\draw [dashed] (-1,3) -- (-1, 0) node[below]{${-1}$};
\draw [dashed] (1,3) -- (1, 0) node[below]{$1$};
\draw [very thick] (-0.95, 3) to[out=-89, in=110] (-0.5, 1.1) to[out=110+180, in=180] (0, 1.2) to[out=0, in=-100] (0.3, 1.4) to[out=80, in=180+70] (0.6, 2) to[out=70, in=-100] (0.9, 3) node[left]{$\gamma(t)$};
\draw (0,1.2) to (0,0) node[below]{$0$} ;
\draw [->] (-1.3, 1) node[left]{$\displaystyle\mc{A}(c_t, t) = \frac{\overline{\mc A}}{2} + \frac{\pi t}{2}$}to (-.8, 1);
\draw [->] (-0.1, 1.6) node[above]{$(0, u(0, t)\!)$} -- (0, 1.25);
\draw [black, fill=black] (-0, 1.2) circle (0.1ex);
\end{tikzpicture}
\caption{The shaded region trapped below \( \gamma(t) := \graph(u(t)\!) \) must contain the hatched trapezium, and so in particular the area of the hatched trapezium must be bounded from above by the area \( \mc{A}(c_t, t) = \mc{A}(1, t)/2 =   \overline{\mc A}/2  + \pi t /2 \) of the shaded region.}
\label{fig:lowerareabdd}
\end{figure}
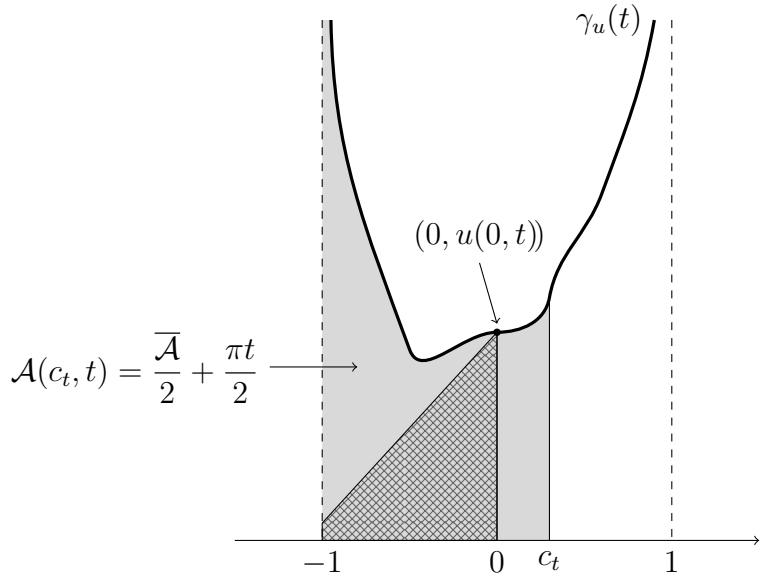

\section{The instantaneous \( \mr{L}^{p>1}_\mr{loc} \)-to-\( \mr{L}^\infty_\mr{loc} \)--estimate}\label{SEC:Lpest}

In this section, we shall use our delayed interior area-to-height estimate (Theorem \ref{thm:intheightest}) to deduce the \emph{instantaneous} \( \mr{L}^{p}_\mr{loc} \)-to-\( \mr{L}^\infty_\mr{loc} \)--estimate of Theorem \ref{thm:intLpest}, where $1 < p < \infty$. 
The proof is essentially a transcription of  \cite[Theorem 4.1]{topping2017sharp}.

\begin{proof}[Proof of Theorem \ref{thm:intLpest}]
As in Claim \ref{claim:suffubdd} in the proof of Theorem \ref{thm:intheightest}, it suffices to assume that \( u(t) \in \mr{C}^{0,1}(J_1) \) for each \( t \in [0, T) \); we omit the details.
Fix \( t \in (0, T) \) and set \( a = \| u_0 \|_{\mr{L}^1(J_1)} = \pi \cdot t_\star \) for \( t _\star := \| u_0 \|_{\mr{L}^1(J_1)} / \pi \).

If \( t \geq a \), then by Theorem \ref{thm:intheightest} with \( 1 + \delta' = \pi \) we have that
\msa{tlarge}
  |u(0, t) | \leq C(\delta' ) + \frac{1}{2} a + \frac{\pi t}{2}\leq C(1 + t) \leq  C \left(1 + \frac{\| u_0 \|_{\mr{L}^p(J_1)}^{p/(p - 1)}}{ t^{1/(p-1)}} + t \right) .
\mfa
If \( t < a \), then let \( k \in \mb{R}_{>0} \) be such that the function $\overline{u}_0 := (|u_0| - k)^+$ satisfies
\ms
\| \overline{u}_0 \|_{\mr{L}^1(J_1)} = \int_{x \in J_1} \overline{u}_0(x) \eds x  = \int_{x \in J_1} (|u_0(x)| - k)^+ \eds x = t. \label{defofk}
\mf
Let \( \overline{u} \in \Cinfloc(J_1 \times (0, \infty)\!) \) be the local GCSF starting from \( \overline{u}_0 \).
Since we have assumed that \( u(t) \in \mr{C}^{0,1}(J_1) \) for all \( t \in [0, \infty) \), the Avoidance Principle (Theorem \ref{thm:avoidii}) tells us that
\ms
|u(x, t)| \leq k + \overline{u}(x, t)
\mf
for all \( (x,t) \in J_1 \times (0, T) \).
Applying Theorem \ref{thm:intheightesteven} to \(\overline{u} \in \Cinfloc(J_1 \times (0, \infty)\!) \) gives that
\ms
|u(0, t)| \leq k + \overline{u}(0, t) \leq k + C({\delta'}) + \frac{1}{2} t + \frac{\pi t}{2}. \label{tsmallk}
\mf
We estimate \( k \in \mb{R}_{>0} \) by integration:
\msa{kest}
\| u_0 \|_{\mr{L}^p(J_1)}^p &= \int_{x \in J_1} |u_0(x)|^p \eds x  \\
&\geq \int_{x \in\{ |u_0| \geq k\}} |u_0(x)|^p \eds x \geq  k^{p-1}  \int_{x \in \{|u_0| \geq k\}} |u_0(x)| \eds x \\
&\geq k^{p-1} \int_{x \in  \{|u_0| \geq k\}} (|u_0(x)| - k) \eds x =  k^{p-1} \int_{x \in J_1} (|u_0(x)| - k)^+ \eds x =  k^{p-1} t,
\mfa
using \eqref{defofk}. 
Hence 
\ms
k \leq \frac{ \| u_0 \|_{\mr{L}^p(J_1)}^{p/(p-1)} }{ t^{1/(p-1)} } \label{estonk},
\mf
which combined with \eqref{tsmallk} gives
\ms\begin{aligned}
 |u(0, t)| \leq \frac{ \| u_0 \|_{\mr{L}^p(J_1)}^{p/(p-1)} }{ t^{1/(p-1)} } + C({\delta}') + \frac{1}{2} t + \frac{\pi t}{2} \leq C \left(1 + \frac{\| u_0 \|_{\mr{L}^p(J_1)}^{p/(p - 1)}}{ t^{1/(p-1)}} + t \right) . 
 \end{aligned} \qedhere \mf
\end{proof}

\section{Existence of solutions starting from \( \mc{M}_*(\mb{R}) \)--initial data}\label{SEC:measgcsf}

In this section, we shall prove Theorem \ref{thm:measgcsf}, that from any \emph{non-atomic real--valued Radon measure \( \nu \in \mc{M}_*(\mb{R}) \)},  there exists a {GCSF} \( u \in \Cinfloc(\mb{R} \times (0, \infty)\!) \) which converges \emph{weakly} to \( \nu \) everywhere and strongly to \( \nu \) away from the singular set of \( \nu \) as time approaches zero. 
See App.~\ref{app:meas} for the definitions in play. \\

Our proof consists of two main steps: the construction of the flow, and the convergence back to the initial condition.

\subsection{Constructing the flow}\label{sec:construction}

We start with a sequence \( (u^n_0)_{n \in \mb{N}} \subset \mr{C}_\mr{c}^\infty(\mb{R}) \) of smooth and compactly supported functions for which \( u^n_0 \msc{L}^1 \rightharpoonup \nu \in \mc{M}(\mb{R}) \) and \( u^n_0 \to u_0 \in \mr{L}^p_\mr{loc}(\Omega) \), and with another sequence  
 \( (U^n_0)_{n \in \mb{N}} \subset \mr{C}_\mr{c}^\infty(\mb{R}) \)  for which \( U_0^n \msc{L}^1 \rightharpoonup |\nu| \in \mc{M}_*^+(\mb{R}) \) and $U_0^n \to |u_0| \in \mr{L}^p_\mr{loc}(\Omega)$, and which is chosen to satisfy \( |u^n_0| \leq U^n_0 \) on \( \mb{R} \) for each \( n \in \mb{N} \).
 We can construct such sequences by cutting-off and mollifying the positive and negative parts of \( \nu \in \mc{M}_*(\mb{R}) \) simultaneously, and then recombining the mollifications (see e.g.~\cite[Ch.~6]{rudin1991functional}).

Let \( u^n \in \mr{C}^\infty(\mb{R} \times [0, \infty)\!) \) be the Ecker--Huisken flow of \( u^n_0 \), and let \( U^n \in \mr{C}^\infty(\mb{R} \times [0, \infty); \mb{R}_{\geq 0}) \) be the Ecker--Huisken flow of \( U^n_0 \). 
The Avoidance Principle (Theorem \ref{thm:avoidii}) applied to \( \pm u^n \) and \( U^n + \eta \) for \( \eta \searrow 0 \) tells us that \( |u^n(t)| \leq U^n(t) \) on \( \mb{R} \) for all \( t \in [0, \infty) \).

We wish to take a locally smooth limit of the flows \( (u^n) \) on \( \mb{R} \times (0, \infty) \); for technical reasons in \S\ref{sec:initialdata}, we will also need to take a limit of the flows \( (U^n) \). 
By the Arzel{\`a}--Ascoli Theorem, it suffices to prove \( \Cinfloc(\mb{R} \times (0, \infty)\!) \)--bounds on the sequences of GCSFs which are independent of \( n \in \mb{N} \).
Recall from \eqref{gcsf'} that  the spatial derivative \( v  = \pd_x u\) of a GCSF $u$ satisfies the equation
\ms
\pd_t v = \pd_x^2 [ \tan^{-1}(v) ] = \pd_x  \left[ \frac{1}{1 + v^2} \pd_x v \right].
\mf  
By  bootstrapping the interior Schauder estimates (see e.g.~\cite[Ch.~III, Theorem 12.1]{ladyzenskaja1988linear}), it suffices to prove \( \mr{C}^{1, \alpha; 0, \frac{\alpha}{2}}_\mr{loc}(\mb{R} \times (0, \infty) \!) \)--bounds which are independent of \( n \).
By the De~Giorgi--Nash--Moser Theorem (see e.g.~\cite[Ch.~III, Theorem 10.1]{ladyzenskaja1988linear}), it suffices to prove \( \mr{C}^{1;0}_\mr{loc}(\mb{R} \times (0, \infty)\!) \)--bounds which are independent of \( n \).
And by Evans \& Spruck's interior gradient estimate (in the form of Theorem \ref{thm:es92cor5.3}), it suffices to prove \( \mr{L}^\infty_\mr{loc}(\mb{R} \times (0, \infty) \! )\)--bounds which are independent of \( n \).

Hence, to construct our GCSF, it suffices to prove

\begin{lem}\label{lem:nuheightbdd}
Let \(  (u^n)_{n \in \mb{N}}, (U^n)_{n \in \mb{N}} \subset \Cinfloc(\mb{R} \times [0, \infty)\!) \) be the sequences of GCSFs as above.
Let \( J \times [\sigma, \tau] \Subset \mb{R} \times (0, \infty) \).
Then there exists a constant \( C = C(\nu, J \times [\sigma, \tau]) < \infty  \) for which
\ms
\| u^n \|_{\mr{L}^\infty(J \times [\sigma, \tau])} \leq \| U^n \|_{\mr{L}^\infty(J \times [\sigma, \tau])} \leq C
\mf
for all sufficiently large \( n \in \mb{N} \).
\end{lem}

We do this now.

\begin{proof}[Proof of Lemma \ref{lem:nuheightbdd}]
We start by scaling and translating our interior area-to-height estimate (Theorem \ref{thm:intheightest}).
Let \( u \in \Cinfloc(J_r(y) \times [0, \infty)\!) \) be a GCSF starting from \( u_0 := u(0) \in \mr{L}^1(J_r(y)\!) \), and consider  the rescaled solution \( u^{r,y} \in \Cinfloc(J_1 \times [0, \infty)\!) \) given by
\ms
u^{r, y}(x, t) := \frac{1}{r} u(y + rx, r^2 t),
\mf
which starts from \( u^{r,y}_0(x) := r^{-1} u_0(y+rx) \), cf.~\eqref{urho}.
We can readily check that
\ms 
\| u_0^{r,y} \|_{\mr{L}^1(J_1)} = \frac{1}{r^2} \| u_0 \|_{\mr{L}^1(J_r(y)\!)}, 
\mf
cf.~\eqref{urhoL1}, and
\ms 
\left| u^{r,y}\left(0, \frac{t}{r^2}\right) \right| = \frac{1}{r} |u(y, t) |,
\mf
cf.~\eqref{urhoLinf}.
Therefore Theorem \ref{thm:intheightest} gives the upper bound
\ms
|u(y, t)| \leq C r + \frac{1}{2r} \| u_0 \|_{\mr{L}^1(J_r(y)\!)} + \frac{\pi t}{2 r} \label{urzupbdd}
\mf
for all
\ms
t \geq \| u_0 \|_{\mr{L}^1(J_r(y)\!)},
\mf
for the constant \( C = C({\delta}') < \infty \) for \( 1 + {\delta}' = \pi \), cf.~\eqref{urhoupbdd}.

Next, suppose that \( u \in \Cinfloc(J_{2r}(x) \times [0, \infty) \! )\) is a GCSF starting from \( u_0 := u(0) \in \mr{L}^1(J_{2r}(x)\!) \). 
For any \(y \in J_r(x) \) and \( \| u_0 \|_{\mr{L}^1(J_{2r}(x)\!)} \leq S \leq T \), we can apply the rescaled estimate \eqref{urzupbdd} to \( u \in \Cinfloc(J_r(y) \times [0, T]) \) to get that
\msa{}
|u(y, t)|  \leq C r + \frac{1}{2r} \| u_0 \|_{\mr{L}^1(J_r(y)\!)} + \frac{\pi t}{2 r} \leq Cr + \frac{S}{2r}  + \frac{\pi T}{2 r}
\mfa
for all
\ms
S \leq t \leq T,
\mf
i.e.~that there holds the bound
\ms
 \| u \|_{\mr{L}^\infty(J_r(x) \times [S, T])} \leq C(r, S, T) \label{ufatintbdd}
\mf
for all \( \| u_0 \|_{\mr{L}^1(J_{2r}(x)\!)} \leq S \leq T  \).
We will apply our interior area-to-height estimate (Theorem \ref{thm:intheightest}) in the form \eqref{ufatintbdd}.

It suffices to prove the upper bound on \( U^n \) only.
For each \( x \in \mb{R} \), let \( r_x = r_x(\nu, \sigma) > 0 \) be such that
\ms
|\nu|(J_{2 r_x}(x)\!) \leq \frac{\sigma}{2}.
\mf
By the weak convergence of \( U^n_0 \) to \( |\nu| \) as \( n \nearrow \infty \) and by the non-atomicity of \( \nu \), it follows (see e.g.~\cite[Theorem 1.40]{evans2015measure}) that 
\ms
\| U^n_0 \|_{\mr{L}^1(J_{2r_x}(x)\!)} = \int_{y \in J_{2 r_x}(x)} U^n_0(y) \eds y   \to \int_{y \in J_{2 r_x}(x)} \eds |\nu|(y)  =  |\nu|(J_{2 r_x}(x)\!) \leq \frac{\sigma}{2}
\mf
as \( n \nearrow \infty \), and so there exists \( n_x = n_x(\nu, \sigma) \in \mb{N} \) for which
\ms
 \| U^n_0 \|_{\mr{L}^1(J_{2r_x(x)})}  \leq \sigma
\mf
for all \( n \geq n_x \).
Hence, by \eqref{ufatintbdd}, we have that
\ms
 \| U^n \|_{\mr{L}^\infty(J_{r_x}(x) \times [\sigma, \tau])} \leq C(r_x, \sigma, \tau)
\mf
for all \( n \geq n_x \).
By the pre-compactness of \( J \Subset \mb{R} \), we can pick a finite collection \( X_J = X_J(\nu, \sigma) = \set{x_1, x_2, \dots, x_m} \subset \mb{R} \) for which 
\ms
J \subset \bigcup_{x \in X_J} J_{r_x}(x).
\mf
Maximising \( n_x = n_x(\nu, \sigma) \) and \( C(r_x, \sigma, \tau ) \) over all \( x \in X_J = X_J(\nu, \sigma) \) gives the desired \( \mr{L}^\infty(J \times [\sigma, \tau]) \)--bound.
\end{proof}
 
We now pass to a common sub-sequence of indices \( (n_1, n_2, n_3, \dots) \) to extract limits \( u \in \Cinfloc(\mb{R} \times (0, \infty) \! ) \) and \( U \in \Cinfloc(\mb{R} \times (0, \infty); \mb{R}_{\geq 0}) \). 
The ordering \( |u(t)| \leq U(t) \) on \( \mb{R} \) for every \( t \in (0, \infty) \) is preserved in the limit. \\

\emph{The construction of the flow is now complete.}

\subsection{Convergence to the initial datum}\label{sec:initialdata}

We continue the proof of Theorem \ref{thm:measgcsf} by showing that the solution \( u \in \Cinfloc(\mb{R} \times (0, \infty)\!) \) constructed in \S\ref{sec:construction} converges to the initial datum in the senses of \eqref{convM}--\eqref{convL1loc}. 

We handle the weak convergence first. 
As in the analogous case of the two--dimensional RF in \cite{topping2021smoothing}, we shall do this via an area estimate.

\begin{prop}[{cf.~\cite[Lemma 4.5]{topping2021smoothing}}]\label{prop:dIntdt}
Let \( u  \in \Cinfloc(J \times (0, T) \! ) \) be a GCSF and let \( \varphi \in \mr{C}_\mr{c}^\infty(J) \).
Then
\ms\label{dIntdt}
\left| \dbyd{}{t} \int_{y\in J} \varphi(y) u(y, \cdot \, ) \eds y \, \right| \leq  C(\varphi)  := \frac{\pi}{2} \int_{y \in J} |\pd_x\varphi(y)| \eds y. 
\mf
\end{prop}

\begin{proof}
Differentiating under the integral and writing \( v = \pd_x u \), we have that
\msa{}
\left| \, \left. \dbyd{}{t} \right. \int_{y \in J} \varphi(y) u(y, \cdot) \eds y \, \right| &=  \left| \, \int_{y \in J} \varphi(y) \cdot \pd_x [ \tan^{-1}(v)](y, s) \eds y \, \right| \\
&= \left| \, \int_{y \in J} \pd_x \varphi(y) \cdot [\tan^{-1}(v)](y, s) \eds y \, \right| \\ 
& \leq \frac{\pi}{2} \int_{y \in J} |\pd_x \varphi(y)| \eds y. 
\end{aligned}\qedhere\mf
\end{proof}

\begin{proof}[Proof of \eqref{convM}]
Let \( (t_m)_{m \in \mb{N}} \subset (0, \infty) \) be an arbitrary sequence with \( t_m  \to 0 \) as \( m \nearrow \infty \). 
We need to show that \eqref{defnwconv} holds with \( \nu_m = u(t_m)\msc{L}^1 \) for all \( \varphi \in \mr{C}^0_\mr{c}(\mb{R}) \).

Suppose for now that \( \varphi \in \mr{C}^\infty_c(\mb{R}) \). 
Applying Proposition \ref{prop:dIntdt} to \( u^n \) and integrating over \( t \in [0, t_m] \) gives that
\ms
\left| \, \int_{y \in \mb{R}} \varphi(y) u^n(y, t_m) \eds y - \int_{y \in \mb{R}} \varphi(y) u^n_0(y) \eds y \, \right| \leq C(\varphi) t_m. \label{weakconvn}
\mf
By taking \( n \nearrow \infty \) in \eqref{weakconvn}, using the local smooth convergence \( u^n(t_m) \to u(t_m) \) and the weak convergence \( u^n_0 \msc{L}^1 \rightharpoonup \nu \), we get that
\ms
\left| \, \int_{y \in \mb{R}} \varphi(y) u(y, t_m) \eds y - \int_{y \in \mb{R}} \varphi(y)  \eds\nu( y )\, \right| \leq C(\varphi) t_m. \label{weakconv}
\mf  
The right-hand side of \eqref{weakconv} clearly converges to zero as \( m \nearrow \infty \).

In particular, by replacing \( u \) with \( U \) and by replacing \( \varphi \) with any \( \chi \in \mr{C}^\infty_\mr{c}(J; \mb{R}_{\geq 0}) \) for which \( C(\chi) = \pi \), we get from \eqref{weakconv} that
\ms
\int_{y \in J} \chi(y) U(y, t_m) \eds y \leq \int_{y \in J} \chi(y) \eds|\nu|(y) + \pi t_m. \label{weakconvU}
\mf
Taking \( \chi \nearrow \chi_J \) gives by the Monotone Convergence Theorem that 
\ms
 \| U(t_m) \|_{\mr{L}^1(J)} \leq |\nu|(J) + \pi t_m . \label{Ugrow}
 \mf 

For a general \( \varphi \in \mr{C}^0_\mr{c}(\mb{R}) \), we approximate by a smooth \( \widetilde\varphi \in \mr{C}^\infty_\mr{c}(J) \) where \( \supp(\varphi) \Subset J \Subset \mb{R} \).
For see that
\msa{tilphitrick}
&\left| \, \int_{y \in \mb{R}} \varphi(y) u(y, t_m) \eds y - \int_{y\in\mb{R}} \varphi(y) \eds\nu(y) \right| \\
&\qquad{} \leq \left| \, \int_{y \in J} [\varphi(y) - \widetilde{\varphi}(y)]  u(y, t_m) \eds y \, \right|  \\
&\qquad\qquad{}+ \left| \, \int_{y \in \mb{R}} \widetilde{\varphi}(y) u(y, t_m) \eds y - \int_{y \in \mb{R}} \widetilde{\varphi}(y) \eds\nu(y) \right| \\
&\qquad\qquad{}+ \left| \, \int_{y \in J} [\widetilde{\varphi}(y) - \varphi(y)] \eds\nu(y) \right|  \\
 &\qquad\leq \| \varphi - \widetilde{\varphi} \|_{\mr{L}^\infty(J)} \big[ \| u(t_m) \|_{\mr{L}^1(J)} + |\nu|(J) \big] \\
 &\qquad\qquad{} + \left| \, \int_{y \in \mb{R}} \widetilde{\varphi}(y) u(y, t_m) \eds y - \int_{y \in \mb{R}} \widetilde{\varphi}(y) \eds\nu(y) \right|.\!
\mfa
The last term in \eqref{tilphitrick} converges to zero as \( m \nearrow \infty \), so it suffices to show that the second-last term can be chosen to be arbitrarily small for all \( m \in \mb{N} \).
By the presence of the factor of \( \|\varphi - \widetilde\varphi\|_{\mr{L}^\infty(J)} \), it suffices to show that \( \| u(t_m) \|_{\mr{L}^1(J)}  + |\nu|(J) \) can be chosen to be bounded for all \( m \in \mb{N} \).
But by \eqref{Ugrow}, this is clearly the case.
\end{proof}

The strong convergence \eqref{convL1loc} is a modification of part of \cite[Theorem A]{chou2020general}; we recall their argument for completeness.
The key tool is the following \( \mr{L}^p_\mr{loc} \)--separation estimate: 

\begin{prop}[{Contained in \cite[Proposition 3.1]{chou2020general}}]\label{prop:c20prop3.1}
Let \( u_1, u_2 \in \Cinfloc(J_R \times (0, T) \! ) \) be two {GCSF}s defined on \( J_R \subset \mb{R} \). 
Then for all \( 0 < r < R \) and \( \delta > 0 \), there exists a cut-off function \( \chi \in \mr{C}_\mr{c}^\infty(J_R; [0,1]) \) with \( \chi = 1 \) on \( J_r \Subset J_R \) such that, for all \( 1 \leq p < \infty \), there holds the estimate
\ms
\| (u_1(t) - u_2(t)\!) \chi \|_{\mr{L}^p(J_R)} \leq \| (u_1(s) - u_2(s)\!) \chi \|_{\mr{L}^p(J_R)} + \frac{ 2 \pi (1 + \delta) p}{(R - r)^{1 - \frac{1}{p}}} (t - s)  \label{Lpest}
\mf
for all \( 0 < s \leq t < T \).
\end{prop}

\begin{proof}[Sketch proof]
Choose \( \chi \in \mr{C}^\infty_\mr{c}(J_R; [0,1]) \) to satisfy \( \chi(y) = 1 \) for all \( y \in J_r \) and \( |\pd_x \chi(y)| \leq  \frac{1 + \delta}{R - r} \chi_{J_R \setminus J_r}(y) \) for all \( y \in J_R \).
Noting that \(   |\!\tan^{-1}(v)| \leq  \pi / 2 \) for all \( v \in \mb{R} \), following the proof of \cite[Proposition 3.1]{chou2020general} leads to \eqref{Lpest}.
\end{proof}

\begin{rem}\label{rem:sharpareagrow}
In the case of \( p = 1 \) and \( R  = 1 \) in Proposition \ref{prop:c20prop3.1}, by taking \( r \nearrow 1 \) and \( \delta \searrow 0 \) in \eqref{Lpest}, we deduce that
\ms
 \| u_1(t) - u_2(t) \|_{\mr{L}^1(J_1)} \leq \| u_1(s) - u_2(s) \|_{\mr{L}^1(J_1)} + 2 \pi (t - s) . \label{globL1ineq}
\mf
By letting \( u_1 \) and \( u_2 \) be disjoint Grim Reaper solutions \eqref{grimreaper} translating away from each other, one can see that \eqref{globL1ineq} is sharp.
\end{rem}

\begin{proof}[Proof of \eqref{convL1loc}]
We shall apply Proposition \ref{prop:c20prop3.1} in following weakened form: {\it For all \(  0 < r < R \)  and \( 1 \leq p < \infty \), there exists a constant \( C_{p, R-r} \) for which
\ms
\| u_1(t) - u_2(t) \|_{\mr{L}^p(J_r)} \leq \| u_1(s) - u_2(s)  \|_{\mr{L}^p(J_R)} + C_{p, R-r}(t - s) \label{Lpestinex}
\mf
for all \( 0 < s \leq t< T \).}

Let \( \mc{K} \Subset \Omega \) be compact; then as \( \Omega \subseteq \mb{R} \) is open, and therefore a countable union of disjoint open intervals, it follows that \( \mc{K} \) is contained within a \emph{finite} union of open intervals. 
Thus, it suffices to prove that \( u(t) \to u_0 \in \mr{L}^p(J) \) for a single open interval \( J \Subset \Omega \). 
By translation, we may suppose that \( J = J_r \Subset J_R \subseteq \Omega \).

Let \( \veps > 0 \).
We first choose \(n \in \mb{N} \) so large that 
\ms
 \| u^n_0 - u_0 \|_{\mr{L}^p(J_R)} \leq \frac{\veps}{4}. 
 \mf 
With this \( n\) fixed, we then choose \( t_\veps > 0 \) so small that
\ms
 \| u^n(t) - u(t) \|_{\mr{L}^p(J_r)} \leq \frac{\veps}{2} \qquad \text{ and }  \qquad \| u^n(t) - u^n_0 \|_{\mr{L}^p(J_r)} \leq \frac{\veps}{4}
\mf
for all \( 0 < t < t_\veps \), using \eqref{Lpestinex} and the locally smooth convergence \( u^n(t) \to u^n_0 \) as \( t \searrow 0 \) respectively.
Taken altogether, we get that
\msa{}
\| u(t) - u_0 \|_{\mr{L}^p(J_r)} &\leq \| u(t) - u^m(t) \|_{\mr{L}^p(J_r)} + \| u^m(t) - u^m_0 \|_{\mr{L}^p(J_r)} + \| u^m_0 - u_0 \|_{\mr{L}^p(J_r)} \\
&\leq \frac{\veps}{2} + \frac{\veps}{4} + \frac{\veps}{4} = \veps
\mfa
for all \( 0 < t < t_\veps \).
\end{proof}

{\it The proofs of the convergences \eqref{convM}--\eqref{convL1loc} is now complete.} \\

\noindent {\it This completes the proof of Theorem \ref{thm:measgcsf}.}

\subsection{Some comments on the work of Chou \& Kwong}\label{sec:choukwong}

First, we clarify that if the initial datum \( \nu = u_0 \msc{L}^1 \in \mc{M}_*(\mb{R}) \) has density  \( u_0 \in \mr{L}^p(\mb{R}) \) for some \( 1\leq p < \infty \), then the solution \( u \in \Cinfloc(\mb{R} \times (0, \infty)\!) \) constructed in Theorem \ref{thm:measgcsf} attains its initial datum in \( \mr{L}^p(\mb{R}) \).

\begin{proof}[Proof of \eqref{convLp}]
We appeal to \eqref{Lpestinex}. 
Recall that \eqref{convL1loc} already shows the \( \mr{L}^p_\mr{loc} \)--convergence.
By setting \( u_1 = u \) and \( u_2 = 0 \) and letting \( s \searrow 0 \), we get that
\ms
\| u(t) \|_{\mr{L}^p(J_r)} \leq \| u_0 \|_{\mr{L}^p(J_R)} + C_{p, R-r} t.
\mf
Letting \( r := R - 1 \nearrow \infty \) gives us that
\ms
\| u(t) \|_{\mr{L}^p(\mb{R})} \leq \| u_0 \|_{\mr{L}^p(\mb{R})} + C_{p, 1} t,
\mf
and so, in particular, \( u(t) \in \mr{L}^p(\mb{R}) \) for all \( t \in [0, \infty) \) and 
\ms
\limsup_{t \searrow 0} \| u(t) \|_{\mr{L}^p(\mb{R})} \leq \| u_0 \|_{\mr{L}^p(\mb{R})}. \label{limsupest}
\mf 

The following measure-theoretic lemma is straightforward to prove (see e.g.~\cite[\S4.3]{sobphd}):

\begin{lem}\label{lem:Lploctoglob}
Let \( (f_m)_{m \in \mb{N}} \subset \mr{L}^p(\mb{R}) \) be a sequence of functions for some fixed \( 1 \leq p < \infty \), and let \( f \in \mr{L}^p(\mb{R}) \).
Suppose that
\ms\notag
f_m \to f \in \mr{L}^p_\mr{loc}(\mb{R}) \qquad \text{ and } \qquad \limsup_{m \nearrow \infty} \| f_m \|_{\mr{L}^p(\mb{R})} \leq \| f \|_{\mr{L}^p(\mb{R})}.
\mf
Then
\ms
f_m \to f \in \mr{L}^p(\mb{R}).
\mf
\end{lem} 

Applying Lemma \ref{lem:Lploctoglob} to \( f_m := u(t_m) \) and \( f := u_0 \) for an arbitrary sequence \( (t_m)_{m \in \mb{N}} \subset (0, \infty) \) converging to zero, keeping in mind \eqref{convL1loc} and \eqref{limsupest}, gives the desired global convergence \eqref{convLp}.
\end{proof}

Second, we highlight that Proposition \ref{prop:c20prop3.1} implies the following continuity property for the GCSF in \( \mr{L}^{p}_\mr{loc} \) for $1 < p < \infty$:

\begin{thm}[Continuous dependence on $\mr{L}^p_\mr{loc}$--initial data when $p >1$]\label{thm:Lploccont}
Let \( 1 < p < \infty \), and let \( u_0 \in \mr{L}^p_\mr{loc}(\mb{R}) \).
Let \( u \in \Cinfloc(\mb{R} \times (0, \infty)\!) \) be any GCSF for which \( u(t) \to u(0) := u_0 \in \mr{L}^p_\mr{loc}(\mb{R}) \) as \( t \searrow 0 \). 
Then \( u \in \mr{C}^0_\mr{loc}([0, \infty); \mr{L}^p_\mr{loc}(\mb{R})\!) \), and the mapping \( \mr{L}^p_\mr{loc}(\mb{R}) \ni u_0 \mapsto  u \in \mr{C}^0_\mr{loc}([0, \infty); \mr{L}^p_\mr{loc}(\mb{R})\!) \) is well-defined and continuous; in particular, for each \( u_0 \in \mr{L}^p_\mr{loc}(\mb{R}) \), the associated solution \( u \in \Cinfloc(\mb{R} \times (0, \infty)\!) \) is unique.
\end{thm}

\begin{proof}
It is immediate that \( u \in \mr{C}^0_\mr{loc}(\!(0, \infty); \mr{L}^p_\mr{loc}(\mb{R})\!) \); the continuity down to the initial time is precisely the statement that the GCSF attains its initial datum in \( \mr{L}^p_\mr{loc}(\mb{R}) \). 

The rest of the theorem follows from Proposition \ref{prop:c20prop3.1}.
Fix \( J_r \Subset \mb{R} \) and \( [a, b] \Subset [0, \infty) \), and let \( \veps > 0 \) be arbitrary.
Choose \( R > 0 \) so large that
\ms
 r < \frac{R}{2}  \qquad \text{ and } \qquad \frac{2^{3 - \frac{1}{p}} \pi p t}{R^{1 - \frac{1}{p}}} \leq \frac{8 \pi p b}{R^{1 - \frac{1}{p}}} \leq \frac{\veps}{2},  \label{errorsmall}
\mf
the latter for all \( t \in [a, b] \).
Suppose that \( \widetilde{u}_0 \in \mr{L}^p_\mr{loc}(\mb{R}) \) is such that
\ms
\| u_0 - \widetilde{u}_0 \|_{\mr{L}^p(J_R)} \leq \frac{\veps}{2},  \label{initdatclose}
\mf
and let \( \widetilde{u} \in \Cinfloc(\mb{R} \times (0, \infty)\!) \) be any GCSF starting from \( \widetilde{u}_0 \) in the \( \mr{L}^p_\mr{loc}(\mb{R}) \)--sense. 
Setting \( \delta =  1 \), \( u^1 = u \) and \( u^2 = \widetilde{u}\) in  Proposition \ref{prop:c20prop3.1}, we get for every non-zero \(  t \in [a, b] \) that
\msa{Lplocclose}
\| u(t) - \widetilde{u}(t) \|_{\mr{L}^p(J_r)} &\leq \| u(t) - \widetilde{u}(t) \|_{\mr{L}^p(J_{R/2})}  \\
&\leq \| u(s) - \widetilde{u}(s) \|_{\mr{L}^p(J_{R})} +  \frac{4\pi p (t - s)}{(R/2)^{1 - \frac{1}{p}}} \\
&\to \| u_0 - \widetilde{u}_0 \|_{\mr{L}^p(J_{R})} +  \frac{2^{3 - \frac{1}{p}} \pi p t }{R^{1 - \frac{1}{p}}} \\
&\leq \frac{\veps}{2} + \frac{\veps}{2} = \veps
\mfa
as \( s \searrow 0 \), using \eqref{errorsmall}--\eqref{initdatclose}.
Taking the supremum over \( t \in [a, b] \) gives that
\ms
\| u - \widetilde{u} \|_{\mr{C}^0([a, b]; \mr{L}^p(J_r)\!)} = \sup_{t \in [a, b]} \| u(t) - \widetilde{u}(t) \|_{\mr{L}^p(J_r)} \leq \veps.
\qedhere 
\mf
\end{proof}

\begin{rem}
While Theorem \ref{thm:Lploccont}, together with \cite[Theorem A]{chou2020general}, fully resolves the question of well-posed-ness of the GCSF from $\mr{L}^p_\mr{loc}$--initial data when $p > 1$ is strictly greater than $1$, the question of uniqueness and continuous dependence from $\mr{L}^1_\mr{loc}$--initial data, i.e.~the case when $p = 1$, is still open. 
 Since Proposition \ref{prop:c20prop3.1} in the case $p=1$ is sharp---see Remark \ref{rem:sharpareagrow}---the author expects a proof of uniqueness from $\mr{L}^1_\mr{loc}$--initial data to require a novel estimate. 
\end{rem}

\section{Towards a correspondence}\label{SEC:correspond}

Recall that our existence procedure in \S\ref{SEC:measgcsf} took a measure \( \nu \in \mc{M}_*(\mb{R}) \) and produced \emph{two} GCSFs:
 one starting from \( \nu \) and the other starting from \( |\nu| \);
we continue to denote these solutions by \( u \in \Cinfloc(\mb{R} \times (0, \infty)\!) \) and \( U \in \Cinfloc(\mb{R} \times (0, \infty); \mb{R}_{\geq 0}) \) respectively. 
These flows were constructed in such a way as to ensure that \( |u(t)| \leq U(t) \) on \( \mb{R} \) for every \( t \in (0, \infty) \).

In this section, we shall use this additional property, that each of the GCSFs generated by Theorem \ref{thm:measgcsf} is \emph{dominated} by a positive GCSF, to further a correspondence between \emph{dominated GCSFs} and \( \mc{M}_*(\mb{R}) \), following on from Chou \& Kwong in \cite[Appendix]{chou2020general}. 
For the analogous correspondence in the RF setting, see \cite[Theorem 1.7]{peachey2024twodimensional}.

\begin{defn}[Dominated GCSF]\label{defn:domgraflow}
We say that a GCSF \( u \in \Cinfloc(\mb{R} \times (0, \infty)\!) \) is \emph{dominated} if there exists a positive GCSF \( U \in \Cinfloc(\mb{R} \times (0, \infty); \mb{R}_{\geq 0}) \) for which
\ms
| u (t) | \leq U (t)
\mf
on \(  \mb{R} \) for all times \( t \in (0, \infty) \).
\end{defn}

We then have

\begin{thm}[{\emph{Partial} correspondence between dominated GCSFs and \( \mc{M}_*(\mb{R}) \)}\!]\label{thm:partialcorrespond}
There holds the following \emph{partial} correspondence between dominated GCSFs and \( \mc{M}_*(\mb{R}) \):
\begin{enumerate}[(i)]
\item For each \( \nu \in \mc{M}_*(\mb{R}) \), there exists {\normalfont at least one} dominated GCSF \( u \in \Cinfloc(\mb{R} \times (0, \infty)\!) \) for which \( u(t) \msc{L}^1 \rightharpoonup \nu \) as \( t \searrow 0 \).
\item \label{initialtrace} For each dominated GCSF \( u \in \Cinfloc(\mb{R} \times (0, \infty)\!) \), there exists a unique \( \nu \in \mc{M}_*(\mb{R}) \), so-called the {\normalfont initial trace} in \cite{chou2020general}, for which \( u(t) \msc{L}^1  \rightharpoonup \nu \in \mc{M}_*(\mb{R}) \) as \( t \searrow 0 \).
\end{enumerate}
\end{thm}

\begin{proof}
We need only prove \eqref{initialtrace}.
For any fixed \( \varphi \in \mr{C}^\infty_\mr{c}(\mb{R}) \), by integrating \eqref{dIntdt} from Proposition \ref{prop:dIntdt} we get that
\ms
\left| \, \int_{x \in \mb{R}} \varphi(x) u(x, t) \eds x - \int_{x \in \mb{R}} \varphi(x) u(x, s) \eds x \, \right| \leq C(\varphi)(t - s) ,
\mf
and therefore the assignment
\ms
\varphi \mapsto L \varphi = \lim_{t \searrow 0} \int_{x \in \mb{R}} \varphi(x)  u(x, t) \eds x : \mr{C}^\infty_\mr{c}(\mb{R}) \mapsto \mb{R}
\mf
is well-defined, and is entirely determined by the flow $u$.
The map \( L \) inherits the linearity of the integral in the limit. 

Let \( \mc{K} \Subset \mb{R} \) be compact, and choose a bump function \( \chi \in \mr{C}_\mr{c}^\infty( \mb{R}; [0,1] ) \) with \( \chi = 1 \) on \( \mc K \).
Then, for any \( \varphi \in \mr{C}_\mr{c}^\infty(\mc K) \) it holds that \( |\varphi| \leq \| \varphi \|_{\mr{C}^0(\mc K)} \chi \) on \( \mb{R} \).
Therefore
\msa{Lbdd}
\left| \, \int_{x \in \mb{R}} \varphi(x) u(x, t) \eds x \,  \right| &\leq \int_{x \in \mb{R}}  \| \varphi \|_{\mr{C}^0(\mc K)} \chi \cdot  U(x, t) \eds x   = \| \varphi \|_{\mr{C}^0(\mc K)} \int_{x \in \mb{R}} \chi(x) U(x, t) \eds x  \\
&\leq \| \varphi \|_{\mr{C}^0(\mc K)} \left( \,\, \int_{x \in \mb{R}} \chi(x) U(x, 1) \eds x + C(\chi) (1 - t) \! \right) \\
&\leq M_{\mc K} \| \varphi \|_{\mr{C}^0(\mc K)} 
\mfa
for every \( t \in (0, 1] \), where \( U \in \Cinfloc(\mb{R} \times (0, \infty); \mb{R}_{\geq 0}) \) is the positive GCSF which dominates \( u \).
Taking \( t \searrow 0 \) in \eqref{Lbdd}  gives that 
\ms
|L \varphi| \leq M_{\mc K} \| \varphi \|_{\mr{C}^0(\mc K)}.
\mf

The Riesz Representation Theorem (\cite[Theorem 1.38]{evans2015measure}) ensures that \( L \) is given by integration against a unique real--valued Radon measure \( \nu \in \mc{M}(\mb{R}) \).
To show that \( \nu \in \mc{M}_*(\mb{R}) \), i.e.~that \(\nu \) is non-atomic, it suffices (by outer regularity) to show for all \( z \in \mb{R} \) and \( \veps > 0 \) that \( |\nu(V)| \leq \veps \) for all sufficiently small and open \( V \ni z \). 
Given a bump function \( \varphi \in \mr{C}^\infty_\mr{c}(\mb{R}; [0,1]) \) around \( z \) with \( C(\varphi) = \pi \), note that
\ms
\left| \, \int_{x \in \mb{R}} \varphi(x) \eds\nu(x) \right| \leq \int_{x \in \mb{R}} \varphi(x) U\left(x, \frac{\veps}{2\pi} \right) \eds x + \frac{\veps}{2}. 
\mf
If we choose \( R_z > 0 \) so small that
\ms
\int_{x \in J_{2R_z}(z)} U\left(x, \frac{\veps}{2\pi} \right) \eds x \leq \frac{\veps}{2},
\mf 
then 
\ms
\left| \,  \int_{x \in \mb{R}} \varphi(x) \eds\nu(x) \right| \leq \int_{x \in J_{2R_z}(z)} U\left( x, \frac{\veps}{2 \pi} \right) \eds x  + \frac{\veps}{2} \leq \veps \label{nusmallaroundp}
\mf 
whenever \( \varphi \) is supported in \( J_{2 R_z}(z) \).
In particular, given any open \(  V \subset J_{R_z}(z) \) containing \( z \), if \( (\chi^j)_{j \in \mb{N}} \subset \mr{C}^\infty_\mr{c}(J_{2 R_z}(z); [0,1]) \) is such that \( \chi^j \to \chi_V \)  pointwise as \( j \nearrow \infty \), then by the Dominated Convergence Theorem one can pass to the limit in \eqref{nusmallaroundp} to deduce the required
\ms
|\nu(V)| = \lim_{j \nearrow \infty} \left| \, \int_{x \in \mb{R}} \chi^j(x) \eds\nu(x) \right| \leq \veps. \qedhere
\mf
\end{proof}

Note that an alternative argument is provided in \cite[Appendix]{chou2020general}, where Chou \& Kwong showed that the limit
\ms
\lim_{t \searrow 0} \int_{x \in \mb{R}} \varphi(x) u(x,t) \eds x 
\mf
in the case \( \varphi \in \mr{C}^1_\mr{c}(\mb{R}) \) can be understood as the distributional derivative of a unique \( \mr{L}^\infty_\mr{loc}(\mb{R}) \)--function. 
And, in the positive case \( u \geq 0 \), 
they similarly characterised the limit as an element of \( \mc{M}_*(\mb{R}) \).\\

It is an open question as to how restrictive our `dominated' condition in Definition \ref{defn:domgraflow} is---see \cite[Question 4.4.17]{sobphd}.

\appendix

\section{Angenent's Intersection Principle}\label{app:avoidinter}

We recall the classical statement of Angenent's Intersection Principle from \cite{angenent1991parabolic}.

\begin{thm}[Classical Angenent's Intersection Principle; {\cite[Variation on Theorem 1.3]{angenent1991parabolic}}]\label{thm:classinter}
Let \( [0, T) \ni t \mapsto \gamma^{1}(t), \gamma^2(t) \subset \mb{R}^2 \) be two distinct CSFs, each defined over either the circle $\mb S^1$ or the closed interval $[0,1]$, for which
\ms
\pd \gamma^1(t) \cap \gamma^2(t) = \gamma^1(t) \cap \pd \gamma^2(t) = \emptyset \label{cptdisjoint}
\mf
for all \( t \in [0, T) \). 
Then the number of intersections of \( \gamma^1(t) \) and \( \gamma^2(t) \) is a finite and decreasing function of \( t \in (0, T) \), which strictly decreases whenever \( \gamma^1(t) \) and \( \gamma^2(t) \) meet tangentially.
\end{thm}

Notice that Theorem \ref{thm:classinter} is concerned with the evolution of \emph{compact} arcs, while in this paper we are primarily be concerned with \emph{non-compact} curves.
This issue can be remedied as follows. \\

First, we restrict our attention to CSFs which admit `good' parametrisations, in the sense of Peachey's \emph{uniformly proper} condition.

\begin{defn}[Uniformly proper; {\cite[Definition 4.2.1]{peachey2022thesis}}]\label{defn:unifprop}
We say that a CSF \( \gamma : \Omega \times [0, T) \mapsto \mb{R}^2 \) is \emph{uniformly proper} if every restriction \( \gamma : \Omega \times [0, t] \mapsto \mb{R}^2 \) for \( 0 < t < T \) is a proper map.
\end{defn}

Peachey's subclass of CSFs possesses the important feature that, Ecker--Huisken flows of locally Lipschitz functions and Chou--Zhu flows of curves whose ends are graphical are automatically uniformly proper---see \cite[Remark 2.2.31 and Corollary 2.4.15]{sobphd}.  
Clearly compact CSFs are automatically uniformly proper.

Second, we introduce an extension of the condition \eqref{cptdisjoint}.

\begin{defn}[Disjoint ends]\label{defn:disends}
We say that two CSFs \( [0, T) \ni t \mapsto \gamma^1(t), \gamma^2(t) \subset \mb{R}^2 \) have \emph{disjoint ends} if \eqref{cptdisjoint} holds and if there exists a compact set \( \mc{K} \Subset \mb{R}^2 \) for which
\ms
\gamma^1(t) \cap \gamma^2(t) \cap (\mb{R}^2 \setminus \mc{K}) = \emptyset
\mf
for all \( t \in [0, T) \).
\end{defn} 

Armed with Definitions \ref{defn:unifprop} and \ref{defn:disends}, we have the following version of Angenent's Intersection Principle, which shall be the main version used in this paper:

\begin{thm}[Angenent's Intersection Principle; see {\cite[Theorem 2.4.10]{sobphd}}]\label{thm:inter}
Let \(  [0, T) \ni t \mapsto  \gamma^1(t), \gamma^2(t) \subset \mb{R}^2 \) be two distinct uniformly proper CSFs, starting from the locally Lipschitz curves \( \gamma^{1}_0 \), \( \gamma^{2}_0 \) respectively, which have disjoint ends.
Then the number of intersections between \( \gamma^1(t) \) and \( \gamma^2(t) \) is finite function of \( t \in (0, T) \), and a decreasing function of \(t \in [0, T) \) which strictly decreases whenever \( \gamma^1(t) \) and \( \gamma^2(t) \) meet tangentially.
Moreover, suppose that \( \gamma^1_0 \) and \( \gamma^2_0 \) intersect \emph{one-sidedly}\footnote{%
An intersection \( q \) of \( \gamma^1_0 \) and \( \gamma^2_0 \), with \( \gamma^1_0 \) of class \( \mr{C}^1 \) near \( q \),  is \emph{one-sided} if,
in a neighbourhood of \( q \) we can write sub-arcs of \( \gamma^{1}_0 \) and $\gamma^2_0$ as ordered graphs over a common smooth arc with the ordering becoming strict towards the end-points.%
}  \( m \) times and otherwise intersect \( n \) times, within regions where at least one of the curves is \( \mr{C}^1 \); then the number of intersections for strictly positive times is at most \( n \).
\end{thm}

\begin{proof}[Sketch proof]
The uniformly proper condition and the assumption of disjoint ends allows one to reduce to the case that \(\gamma^1(t) \) and \(\gamma^2(t) \) are both compact for each $t \in [0, T)$, at which point the classical Angenent's Intersection Principle (Theorem \ref{thm:classinter}) applies.  
The Strong Maximum Principle implies that {one-sided} initial intersections disappear instantaneously. 
If there are infinitely many intersections which are not one-sided, then the final assertion is trivial. 
For finitely many initial intersections that are not one-sided, a closer look at Angenent's \cite[Lemmas 5.3 and 5.5]{angenent1988zero} reveals that each one of these must not bifurcate along the flow for a short length of time.  
For details, see \cite[Theorem 2.4.10]{sobphd}.
\end{proof}

An immediate consequence of Angenent's Intersection Principle is the following version of the Avoidance Principle:

\begin{thm}[Avoidance Principle; see {\cite[Theorem 2.4.9]{sobphd}}]\label{thm:avoidii}
Let \( [0, T) \ni t \mapsto \gamma^{1}(t), \gamma^{2}(t) \subset \mb{R}^2 \) be two uniformly proper CSFs, starting from the locally Lipschitz curves \( \gamma^{1}_0 \), \( \gamma^{2}_0 \) respectively, which have disjoint ends.
Suppose that \( \gamma^1_0 \) and \( \gamma^2_0 \) only intersect one-sidedly within regions where at least one of the curves is \( \mr{C}^1 \).
Then for strictly positive times \( t > 0 \) the curves \( \gamma^1(t) \) and \( \gamma^2(t) \) are disjoint.
\end{thm}

\section{Real--valued Radon measures}\label{app:meas}

We recall the various measure-theoretic notions required in this paper, mostly from \cite{evans2015measure}. \\

By a \emph{real--valued Radon measure} (\cite[\S1.8]{evans2015measure}) on \( \mb{R} \), we mean a Radon measure \( \mu \) on \( \mb{R} \) paired with a \( \mu \)--measurable function \( \sigma : \mb{R} \mapsto \set{\pm 1} \). 
We write \( \nu = \sigma \mu \) and \( |\nu| = \mu \). 
We denote the space of real--valued Radon measures by \( \mc{M}(\mb{R}) \) and the space of Radon measures by \( \mc{M}^+(\mb{R}) \). \\

The Riesz Representation Theorem (\cite[Theorem 1.38]{evans2015measure}) identifies \( \mc{M}(\mb{R}) \) with the dual space \( \mr{C}^0_\mr{c}(\mb{R})^* \) of continuous and compactly supported functions.
A sequence \( (\nu_n)_{n \in \mb{N}} \) of measures \emph{converges weakly}  (\cite[\S 1.9.1]{evans2015measure}) to another measure \( \nu \), written \( \nu_n \rightharpoonup \nu \), as \( n \nearrow \infty \) if it converges in the weak--$*$ sense when viewed in  \( \mr{C}^0_\mr{c}(\mb{R})^* \); explicitly, 
\ms
\int_\mb{R} \varphi \eds\nu_n \to \int_\mb{R} \varphi \eds \nu  \label{defnwconv}
\mf
as \( n \nearrow \infty \) for any \( \varphi \in \mr{C}^0_\mr{c}(\mb{R}) \).  
\\

We say that a real--valued Radon measure \( \nu \) is \emph{non-atomic} if \( \nu(\set{x}) = 0 \) for all \( x \in \mb{R} \); 
note that Chou \& Kwong call such a measure \emph{continuous} in \cite{chou2020general}.
We denote the space of non-atomic real--valued Radon measures by \( \mc{M}_*(\mb{R}) \), and the space of non-atomic Radon measures by \( \mc{M}^+_*(\mb{R}) \). Clearly \( \nu \in \mc{M}_*(\mb{R}) \) if and only if \( |\nu| \in \mc{M}^+_*(\mb{R}) \).

\section{Evans \& Spruck's interior gradient estimate}\label{app:es92cor5.3}

We will use the following version of Evans \& Spruck's interior gradient estimate: 

\begin{thm}[cf.~{\cite[Corollary 5.3]{evans1992motion}}]\label{thm:es92cor5.3}
Let \( u  \in \Cinfloc(J_{2R} \times [0, 2T]) \) be a GCSF for which \( \| u \|_{\mr{L}^\infty(J_{2R} \times [0, 2T])} \leq L \). 
Then, for any \( 0 < S < T \), there exists a constant \(  C(R, S, T, L) < \infty \) such that
\ms \label{es92cor5.3}
 \| \pd_x u \|_{\mr{L}^\infty(J_R \times [S, T])} \leq C( R, S, T, L ).
\mf
\end{thm}

\begin{proof}
Evans \& Spruck's estimate is phrased in \cite[Corollary 5.3]{evans1992motion} as
\ms
| \pd_x u(0, s) | \leq C \left(R, s, \| u \|_{\mr{L}^\infty(J_R \times [0, 2s])} \right)  \label{es92cor5.3orig}.
\mf
To obtain \eqref{es92cor5.3}, we apply \eqref{es92cor5.3orig} to \( J_R \times [0, 2s] \ni (x, t) \mapsto \widetilde{u}(x, t) = u(y + x, t) \in \mb{R} \) for every \( (y,s) \in J_R \times [S, T] \). 
\end{proof}

\phantomsection
\addcontentsline{toc}{section}{References}
\bibliography{refs}{}
\bibliographystyle{amsalpha}

\end{document}